\documentclass[final]{siamltex}
\usepackage{amssymb}
\usepackage{amsfonts}
\usepackage{amsmath,amsfonts}
\usepackage{graphicx}
\usepackage{algorithm,algorithmic}
\usepackage{cancel} % for cancelto function
\usepackage{tikz}
\usepackage{ifthen}
\usepackage{amsmath}
\usepackage{tikz}
\usepackage[fancy]{tikz-inet}
\usetikzlibrary{shapes,arrows}
\usetikzlibrary{arrows,calc,intersections}
\usetikzlibrary{trees,snakes}
\usetikzlibrary[decorations.pathreplacing]

\newcommand{\highlightedBox}[1]{ #1}
\usepackage[bordercolor=white,backgroundcolor=gray!30,linecolor=black,colorinlistoftodos]{todonotes}
\newtheorem{rem}{Remark}

\newtheorem{prop}{Proposition}
\newtheorem{defi}{Definition}
\newtheorem{coro}{Corollary}
\newtheorem{assu}{Assumption}
\newtheorem{thm}{Theorem}
\newtheorem{lem}{Lemma}

\author{A.F. Hanif$^\dag$, H. Tembine$^\ddag$, M. Assaad$^\ddag$, D. Zeghlache$^\dag$
}
\begin{document}

\title{On the Convergence of a Nash Seeking Algorithm with Stochastic State Dependent Payoffs\thanks{A preliminary work \cite{SPAWC2012} (a conference paper of 5 pages) that focuses on the telecommunications application of some of the results without proofs has been presented at SPAWC 2012. All the theoretical results (i.e. theorems and proofs) are presented in this paper. The research of Ahmed Farhan Hanif and Djamal Zeghlache has received partial funding from the European Community's Seventh Framework Programme (FP7/2007-2013) under grant agreement SACRA n° 249060.  $^\dag$Institut Mines-T\'el\'ecom, T\'el\'ecom SudParis, RS2M Dept, France, $^\ddag$Telecomm Dept, Ecole Superieure d'Electricite (Supelec), France.  $^\dag$\{ahmedfarhan.hanif,djamal.zeghlache\}@it-sudparis.eu, $^\ddag$\{hamidou.tembine,mohamad.assaad\}@supelec.fr.}}

\maketitle

\begin{abstract}
Distributed strategic learning has been getting attention in recent years. As  systems become distributed finding Nash equilibria in a distributed fashion is becoming more important for various applications. In this paper, we develop a distributed strategic learning framework for seeking Nash equilibria under stochastic state-dependent payoff functions. We extend the work of Krstic et.al. in \cite{Kristic}  to the case of stochastic state dependent payoff functions. We develop an iterative distributed algorithm for  Nash seeking and examine its convergence to a limiting trajectory defined by an Ordinary Differential Equation (ODE). We show convergence of our proposed algorithm for vanishing step size and provide an error bound for fixed step size. Finally, we conduct a stability analysis and apply the proposed scheme in a generic wireless networks. We also present numerical results which corroborate our claim.
\end{abstract}

%>>>>>>>>>>>>>>>>>>>>>>>>> For Testing purposes [to be removed before submission]
%\pagebreak
%\tableofcontents
%%\listoftables
%\listoffigures
%\pagebreak
%\theoremlisttype{allname}
%\listtheorems{thm,lem,coro,rem, prop,proof,defi}
%\listtheorems{thm,,coro,defi}
%\listtheorems{thm}
%\listtheorems{coro}
%\listtheorems{defi}
%\pagebreak
%>>>>>>>>>>>>>>>>>>>>>>>>> For Testing purposes [to be removed before submission]
\section{Introduction}
\label{Introduction}
%Stochastic optimization has become a common techniques for finding local extremum or global extremum depending on the function.
%In this paper we consider cases where we don't have a global view of the system.
In this paper we consider a fully distributed system, which consists of non cooperative nodes which can be modeled as a non cooperative game for Nash seeking.
%where there is limited exchange of information between the nodes and the dynamic environment.
% are gaining ground as we are moving toward distributed systems.
Let us consider a distributed system with $N$  nodes or agents which interact with one another and each has a payoff/utility/reward to maximize. The decision or action of each node has an impact on the reward of the other nodes, which makes the problem challenging in general. In such systems each node has access to a numerical value of their utility/reward at each time. In such systems it might not be possible to have a bird's eye view of the system as it is too complicated or is constantly changing.
Let $a_{j,k}$ be the action of node $j$ at time $k$ and the numerical value of the utility of this node is given by $\tilde{r}_{j,k}.$ Where
$\tilde{r}_{j,k}=r_{j}(\mathbf{S}_k,\mathbf{a}_k)+\eta_{j,k}$ were $\eta_{j,k}$ represents noise , $r_j:\mathcal{S}\times \mathbb{R}_{+}^N \longrightarrow \mathbb{R}$ is the payoff function of node $j$, $\mathbf{S}_k\in \mathcal{S}\subseteq \mathbb{C}^{N\times N}$ is the state such that $\mathcal{S}$ is compact, $\mathbf{a}_k=(a_{1,k},\ldots,a_{N,k})$ is the action vector containing actions of all nodes at time $k.$
Figure \ref{Nodesdiagram} shows the system model where we have $N$ interacting nodes.
%Each of these nodes has a reward $\tilde{r}_{j,k}$ and an action $a_{j,k}$.
The rewards are interdependent as the nodes interact with one another. The only assumption that we can make here is the existence of a local solution. Each of these nodes $j$ has access to the numerical value of their respective reward $\tilde{r}_{j,k}$ and it needs to implement a scheme to select an action $a_{j,k}$ such that its utility is maximized. The above scenario can be interpreted as an interactive game. In this paper we explore learning in such games which is synonymous with designing distributed iterative algorithms that converge to the Nash equilibrium.

    \begin{figure*}[!hbtp]
        \begin{center}
            \tikzstyle{block_pu} = [draw, draw=black!90,thick,fill=black!20,rectangle, minimum height=16.5em, minimum width=6em,rounded corners=2em]

%\tikzstyle{block_su} = [draw,dashed, draw=black!90!black,thick,rectangle, minimum height=6em, minimum width=18em,rounded corners=.5em]

\tikzstyle{antenna} = [draw, draw=black!50,thick,fill=black!20, isosceles triangle,rotate=-90,scale=.4 ]

\tikzstyle{sum2}=[ node distance=2cm]

%\tikzstyle{antenna} = [draw, draw=black!50,thick,fill=black!10,
%triangle, minimum height=6em, minimum width=3em]
\tikzstyle{sum}=[draw,draw=black!90,thick,double, fill=black!10, circle , node distance=2cm]
\tikzstyle{input} = [coordinate]
\tikzstyle{output} = [coordinate] \tikzstyle{pinstyle} = [pin edge={to-,thin,black}]

%\pgfdeclarelayer{background layer} \pgfdeclarelayer{foreground
%layer} \pgfsetlayers{background layer,main,foreground layer}

\begin{tikzpicture}%[auto, node distance=6cm,>=latex']

\foreach \r in {1,3,4}
{
    \ifnum \r=3
    \node [sum] (r\r) at (6.5em,-4*2.5 em) {$Rx_\r$};
    \else
    \node [sum] (r\r) at (6.5em,-4*\r em) {$Rx_\r$};
    \fi

    \foreach \t in {1,3,4}
    {
        \ifnum \t=3
            \node [sum] (t\t) at (-3.5em,-4*2.5 em) {$x_j$};
            \node [sum2] (a1) at (-3.5em,-4*1.5 em) {$\vdots$};
            \node [sum2] (a2) at (-3.5em,-4*3 em) {$\vdots$};
        \else
            \ifnum \t=4
                \node [sum] (t\t) at (-3.5em,-4*\t em) {$x_N$};
            \else
                \node [sum] (t\t) at (-3.5em,-4*\t em) {$x_\t$};
            \fi
        \fi

    \ifnum \r=\t

       \ifnum \r=4
%            \node [antenna] (t\t) at (-6.5em,-6*\t em) {$Tx_\t$};
            \draw [black!90, -stealth ,shorten >=.25em,thick,bend right=45,looseness=1] (t\t.east) -- node [above] {${a}_{N, k}$} (r\r.west) ;
            \draw [dashed, color=black!50!black, -stealth,shorten >=.25em,thick, bend right, looseness=1] (r\r.north) to node [below] {$\tilde{r}_{N,k}$} (t\t.north)  ;
        \else

         \ifnum \r=3
%            \node [antenna] (t\t) at (-6.5em,-6*\t em) {$Tx_\t$};
            \draw [black!90, -stealth ,shorten >=.25em,thick,bend right=45,looseness=1] (t\t.east) -- node [above] {${a}_{j, k}$} (r\r.west) ;
            \draw [dashed, color=black!50!black, -stealth,shorten >=.25em,thick, bend right, looseness=1] (r\r.north) to node [below] {$\tilde{r}_{j,k}$} (t\t.north)  ;
            \else

            \draw [black!90, -stealth ,shorten >=.25em,thick,bend right=45,looseness=1] (t\t.east) -- node [above] {${a}_{\t, k}$} (r\r.west) ;
            \draw [dashed, color=black!50!black, -stealth,shorten >=.25em,thick, bend right, looseness=1] (r\r.north) to node [below] {$\tilde{r}_{\t,k}$} (t\t.north)  ;
            \fi
        \fi

    \else
    \fi

    }
}
\node [block_pu] (r_block) at (11em,-4*2.4em) {$
                                                  \begin{array}{cl}
                                                    \text{\textbf{Dynamic}} & \text{\textbf{Environment}}\\
                                                    &\\
                                                    \mathcal{A}_j & \text{action}\\
                                                    \mathcal{S} & \text{state}\\
                                                    \{r_{j}(.)\} & \text{payoff  function}
%                                                    \{{a}_{j,k}\} & \text{action}\\
%                                                    \{{s}_{j,k}\} & \text{state}\\
%                                                    \{r_{j,k}(.)\} & \text{payoff  function}

                                                  \end{array}
                                                $};
%    \node [block_pu, name=pu] at (0,-5em) {};
%    \node [above,black!50] at (pu.north) {Transmitter Receiver 1};
%
%    \node [block_su] (su) at (0,-13em) {};
%    \node [below,black!50!black]at (su.south) {Transmitter Receiver 2};
\end{tikzpicture}
            \caption[Nodes with actions and utilities]{\label{Nodesdiagram}Nodes interacting with each other through a dynamic environment}
        \end{center}
    \end{figure*}
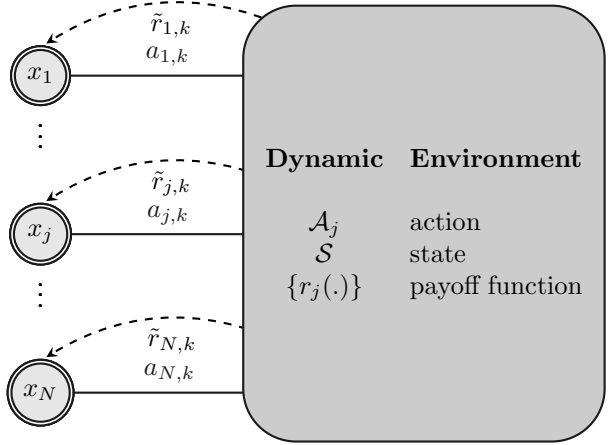

$\bullet$  Different approaches, mainly based on gradient descent or ascent method \cite{Snyman2005}, have been developed to achieve a local optimum (or global optimum in some special cases, e.g. concavity of the payoff, etc.) of the distributed optimization problem.
%From here on out we will use the term gradient ascent as we are dealing with a payoff/reward function that we would like to maximize. %
%Gradient ascent is a well investigated  algorithm from \cite{Snyman2005} for finding the nearest local minima of a function which presupposes that the gradient of the function can be computed.
The method of gradient ascent is also called steepest ascent method, which starts at a point $a_0$ and, as many times as needed, moves from $a_k$ to $a_{k+1}$ by maximizing along the line extending from $a_k$ in the direction of $\nabla r_{j}(\mathbf{S}_k,\mathbf{a}_k)$, the local downhill gradient. This gives the iterative scheme  $ a_{j,k+1}=a_{j,k}+\lambda_k \nabla r_j(\mathbf{S}_k, \mathbf{a}_k)$ where $\lambda_k>0$ is a learning rate/step size. For the applicability of the above algorithm it is necessary to have access to the value of  $\nabla r_{j}(.)$ at each time $k$.
%There are several critical issues that need to be addressed for the applicability of this type of algorithm in games with continuous action spaces as detailed in the next paragraph.
The action can be positive  and upper bounded by a certain maximum value $a_{j,\max}>0$ for some engineering applications. Thus, the component  $a_{j,k}$ needs to be projected in the domain $[0,a_{j,\max}].$ This leads to a projected gradient descent or ascent algorithms: $ a_{j,k+1}= \mbox{proj}_{[0,a_{j,\max}]}\left\{a_{j,k}+\lambda_k \nabla r_{j}(\mathbf{S}_k,\mathbf{a}_k)\right\}$ where $\mbox{proj}$ denotes the projection operator. At each time $k$, node $j$ needs to observe/compute the gradient term  $\nabla r_{j}(\mathbf{S}_k,\mathbf{a}_k).$  Use of the aforementioned gradient based method requires the knowledge of (i)  the system state, (ii) the  actions of others and their states and or (iii) the mathematical  structure (closed form expression) of the payoff function. As we can see, it will be difficult for node $j$ to compute the gradient if the expression for the payoff function $r_j(.)$ is unknown and/or if the states and actions of other nodes are not observed as $r_j(.)$ depends on the actions and states of others.

$\bullet$ There are several methods for Nash equilibrium seeking where we only have access to the numerical value of the function at each time and not its gradient (e.g.  Complex functions which cannot be differentiated or unknown functions). Some of them are detailed below.

 The stochastic gradient ascent proposes to feedback the numerical value of gradient of reward function $\nabla r_{j}$  of node $j$ (which can be noisy) to itself.  This supposes in advance that a noisy gradient can be computed or is available at each node. Note that if  the numerical value of the  gradients of the payoffs are not known by the players, this scheme cannot be used.
 In \cite{Bianchi2013} projected stochastic gradient based algorithm is presented.
 A distributed asynchronous stochastic gradient optimization algorithms is presented in \cite{Bertsekas1986}.
 Incremental Sub-gradient Methods for Non-differentiable Optimization are discussed in \cite{Nedic2001}.
%x
    %Based on these previous works, we pose
%    \begin{quest}
%    Can we extend the above techniques based only on numerical measurements / observations of the payoff functions (not the gradient) and with only basic arithmetic operations at the nodes?
%    \end{quest}
%    The answer to this question is positive for many class of finite games with discrete (and finite) action space.
%\item {Learning for discrete action space:}
A distributed Optimization algorithms for sensor networks is presented in \cite{Rabbat2004}. Interested readers are referred to a survey by Bertsekas \cite{Bertsekas2010} on Incremental gradient, subgradient, and proximal methods for convex optimization. In  \cite{Stankovic2009acc} the authors present Stochastic extremum seeking with applications to mobile sensor networks.

$\bullet$  Krstic et.al. in recent years have contributed greatly to  the field of \textit{non-model based} extremum seeking. In \cite{Kristic}, the authors propose a Nash seeking algorithm for games with continuous action spaces. They proposed a fully distributed learning algorithm and requires only a measurement of the numerical value of the payoff. Their scheme is based on sinus perturbation (i.e. deterministic perturbation instead of stochastic perturbation) of the payoff function in continuous time. However,  discrete time learning scheme with sinus perturbations is not examined in \cite{Kristic}. In \cite{Stankovic2010} extremum seeking algorithm with sinusoidal perturbations for \textit{non-model based} systems has been extended and modified to the case of i.i.d. noisy measurements and vanishing sinus perturbation, almost sure convergence to equilibrium is proved. \textit{Sinus perturbation} based extremum seeking for \textit{state independent noisy measurement}  is presented in \cite{Stankovic2012}. Kristic et al. \cite{stochasticKrstic} have recently extended Nash seeking scheme to stochastic non-sinusoidal perturbations. In this paper we extend the work in \cite{Kristic} to the case of \textit{stochastic state dependent} payoff functions, and use deterministic perturbations for Nash seeking. One can see easily the difference between this paper and the previous existing works \cite{Stankovic2010}\cite{Stankovic2012}. In these works, the noise $\eta_j$ associated with the measurement is i.i.d. which does not hold in practice especially in engineering application where the noise is in general time correlated. In our case, we consider a stochastic state dependent payoff function  and our problem can be written in Robbins-Monro form with a  Markovian (correlated) noise given by  $\eta_j=r_j(\mathbf{S},\mathbf{a})-\mathbb{E}_\mathbf{S}[r_j(\mathbf{S},\mathbf{a})]$ (this will become clearer in the next sections),  i.e.  the associated noise is stochastic state dependent which is different from the case of i.i.d. noise.
%In this paper, discrete time extremum seeking algorithms with sinusoidal perturbations have been developed for three different problems involving autonomous vehicle planar control: a) control of velocity actuated vehicles; b) control of
%force actuated vehicles; c) control of nonholonomic vehicles (unicycles). The algorithms assume time varying gains and are
%able to cope with stochastic perturbations. Convergence to the extremal point, with probability one, has been demonstrated for
%all three cases. It is also shown how the proposed algorithms can be applied to mobile sensors as a tool for achieving
%optimal observation positions. The proposed algorithms have been illustrated with several simulations.

Although stochastic estimation techniques do estimate the gradient but they introduce a level of uncertainty, to avoid this it is possible to introduce sinus perturbation instead of stochastic perturbation. This is particularly  helpful when one node is trying to follow the actions of the other nodes in a certain application.
%There are two directions that try to answer question {\bf Q1} for continuous action game. The first one consists in estimating the gradient based on the measurements. In this case, however, the estimation of the gradient may need itself a long-run interaction, which makes this direction not very suitable for use in practice. The second direction focuses on local perturbation of perceived payoffs (numerical values) to provide an approximated optimal action strategy. In \cite{Kristic}, the authors  propose a Nash seeking algorithm for  games with continuous action spaces. Their scheme is based on sinus perturbation of the payoff function in continuous time.
{%However, the work in \cite{Kristic} considers that
%\begin{enumerate}
%  \item the payoff function is deterministic and state-independent or fixed state whereas in many applications  the payoff can be state-dependent and state can be stochastic.
%  \item the time is continuous whereas algorithms work in discrete time.
%\end{enumerate}}
%This paper address both of the above shortcomings as explained in the contribution section below.
\subsection{Contribution}
%Our work is similar to the work by \cite{Kay} where Kay et.al. try to reduce the probability of error using 'stochastic resonance'.
%Our contribution is based on combining the results from Krstic \cite{Kristic} which represent the stability of sinus estimation and extending
%them to case where the payoff function is dependent on stochastic variables.
%Here we undertake a similar perturbation scenario but instead of stochastic perturbation we use a deterministic sinus function as used by Krstic
%in \cite{Kristic}. We calculate the error bound between our algorithm and the local optimum. %We also calculate the convergence time required for our algorithm.
In this paper, we propose a {\it discrete time } learning algorithm, using sinus perturbation, for {\it continuous}  action games where each node has only a numerical realization of the payoff at each time.  We therefore extend the classical Nash Seeking with sinus perturbation method \cite{Kristic} to the case of  discrete time and stochastic state-dependent payoff functions. We prove that our algorithm  converges locally to a state independent Nash equilibrium in Theorem \ref{Theorem1} for vanishing step size and provide an error bound in Theorem \ref{Theorem2} for fixed step size. Note that since the payoff function may not necessarily be  concave, finding a global optimum at affordable complexity can be difficult in general even in deterministic case (fixed state) and known closed-form expression of payoff. We also show the convergence time for the sinus framework in Corollary \ref{Corollary1}. In this paper we analyze and prove that the algorithm converges to a limiting ODE. We provide the convergence time and error bound between our discrete time algorithm and the ODE.

%{\it To the best of our knowledge this is the first attempt for learning with unknown payoff function in continuous power control for wireless systems.}

%The main contributions are presented in the form of following Theorem \ref{Theorem1}, Theorem \ref{Theorem2}, Theorem \ref{Theorem3}, and
%Theorem \ref{Theorem4} Corollary \ref{Corollary1} and Corollary \ref{Corollary2}.

%\listtheorems{thm,coro}
The proof of the theorems are given in  Appendix \ref{appendix_Convergence}.

\subsection{Structure of the paper}
The remainder of this paper is organized as follows. Section \ref{Problem_Formulation} provides the proposed distributed stochastic learning algorithm. The performance analysis of the proposed algorithm (convergence to ODE, error bounds) is presented in section \ref{Mainresults}. A numerical example with convergence plots is provided in section \ref{NumericalExample}. Section \ref{Conclusion} concludes the paper. Appendix contains the proofs.

\subsection{Notations}

We summarize some of the notations in Table \ref{table_of_notations}.
\begin{table}[h]
\caption{Summary of Notations} \label{table_of_notations}
\begin{center}
\begin{tabular}{ll}
\hline
  Symbol & Meaning \\ \hline
  $\mathcal{N}$ & set of nodes \\
  $\mathcal{A}_j$ &  set of choices of node $j,$  \\
  $\mathbf{S}$ & state space \\
 $ r_{j}$ & payoff of node $j$\\
 $a_{j,k}$ & decision  of $j$ at time $k$\\
 $\mathbf{a}_{-j,k}$ & $(a_{j',k})_{j'\neq j}$\\
 $\mathbb{E}$ & expectation operator\\
 $\nabla$ & gradient operator
  \\ \hline
\end{tabular}
\end{center}
\end{table}
\section{Problem Formulation and Proposed Algorithm\label{Problem_Formulation}}

Let there be $N$ distributed nodes each with a payoff function represented by $r_{j}(\mathbf{S}_k,{a}_{j,k},\mathbf{a}_{-j,k})$ at time $k$ which is used to formulate the following robust problems:
%\subsection{Proposed Solution \label{Solution}}
\begin{eqnarray}\label{eqmaxreward}
\highlightedBox{\sup_{a_{j}\geq0} \mathbb{E}_{\mathbf{S}} r_{j}(\mathbf{S},a_j,\mathbf{a}_{-j})\; \forall \; j\in \mathcal{N}\triangleq\{1,\ldots,N\}}
\end{eqnarray}
A solution to the problem (\ref{eqmaxreward}) is called {\it state-independent} equilibrium solution.

\begin{defi} [Nash Equilibrium (state-independent)]\label{definition1}
$\mathbf{a}^*=(a_{j}^*,\mathbf{a}^*_{-j})\in  \prod_{j'}\mathcal{A}_{j'}$ is a (state-independent) Nash equilibrium point if
\begin{eqnarray}
{\mathbb{E}_{\mathbf{S}}r_j(\mathbf{S},a_j^*,\mathbf{a}^*_{-j})\geq \mathbb{E}_{\mathbf{S}}r_j(\mathbf{S},a'_j,\mathbf{a}^*_{-j}),\ \forall a'_j\in\mathcal{A}_j,\;a'_j\neq a_j^*}
\end{eqnarray} where $\mathbb{E}_{\mathbf{S}}$ denotes the mathematical expectation over the state.
\end{defi}

\begin{defi}[Nash Equilibrium (state-dependent)]\label{definition2}
We define a state-dependent strategy $\tilde{a}_j$ of a node $j$  as a mapping from $\mathcal{S}$ to the action space $\mathcal{A}_j.$ The set of state-dependent strategy is $\mathcal{PG}_j:\ \{\tilde{a}_j:\  \mathcal{S} \longrightarrow \mathcal{A}_j,\ \mathbf{S}\longmapsto \tilde{a}_j(\mathbf{S})\in\mathcal{A}_j\}.$
$$\tilde{a}^*=(\tilde{a}_{j}^*,\tilde{\mathbf{a}}^*_{-j})\in \prod_{i}\mathcal{PG}_i$$ is a (state-dependent) Nash equilibrium point if
\begin{eqnarray}
{\mathbb{E}_{\mathbf{S}}r_j(\mathbf{S},\tilde{a}_j^*(\mathbf{S}),\tilde{\mathbf{a}}^*_{-j}(\mathbf{S}))\geq \mathbb{E}_{\mathbf{S}}r_j(\mathbf{S},\tilde{a}'_j(\mathbf{S}),\tilde{\mathbf{a}}^*_{-j}(\mathbf{S})),\ \forall \tilde{a}'_j\in\mathcal{PG}_j}
\end{eqnarray}
\end{defi}
Here we define $\mathbf{a}:=(a_j,\mathbf{a}_{-j})$
%Assuming that it has any of the limitations mentioned in section \ref{Introduction2}
Assuming that  node $j$  has access to it's realized payoff at each time $k$  but the closed-form expression of $r_j(\mathbf{S}_k, {a}_{j,k},\mathbf{a}_{-j,k})$ is unknown to node $j.$
%such that it is not possible to evaluate its gradient at the receiver.
%Our proposed solution to the above problem only requires local stability at $\mathbf{a}^*$ where $\mathbf{a}^*=(a^*_1,a^*_2,\ldots,a^*_t)$.
A solution to the above problem is a state-independent equilibrium in the sense no node has incentive to change its action when the other nodes keep their choice. It is well-known that equilibria can be different than global optima, the gap between the worse equilibrium and the global maximizer is captured by the so-called {\it price of anarchy}. Thus solution obtained by our method can be suboptimal with respect to maximizing the sum of all the payoffs. We study the local stability of the stochastic algorithm.

The robust game is defined as follows: $\mathcal{N}$ is the set of  nodes, $\mathcal{A}_j$ is the action space of node $j$. $\mathcal{S}$ is the state space of the whole system, where $\mathcal{S} \subseteq \mathbb{C}^{N\times N};$ and $r_j:\ \mathcal{S}\times \prod_{j'\in\mathcal{N}}\mathcal{A}_{j'} \longrightarrow \mathbb{R}$ is a smooth function. It should be mentioned here for clarity that the decisions are taken in a decentralized fashion by each node. Let us continue by stating that  $\mathcal{N}$ is the set of  nodes, $\mathcal{A}_j$ is the action space of node $j$,  $\mathcal{S}$ is the state space of the whole system, where $\mathcal{S}\ \subseteq \mathbb{C}^{N\times N}$ and  $r_j:\ \mathcal{S}\times \prod_{j'\in\mathcal{N}}\mathcal{A}_{j'} \longrightarrow \mathbb{R}.$
%
%\begin{itemize}
%\item $\mathcal{N}$ is the set of  nodes.
%\item $\mathcal{A}_j$ is the action space of node $j$
%\item $\mathcal{S}$ is the state space of the whole system, where $\mathcal{S}\ \subseteq \mathbb{C}^{N\times N}$
%\item $r_j:\ \mathcal{S}\times \prod_{j'\in\mathcal{N}}\mathcal{A}_{j'} \longrightarrow \mathbb{R}$
%\end{itemize}

Games with uncertain payoffs are called robust games. Since state can be stochastic, we get a robust game. Here we will focus on the analysis of the so-called {\it expected robust game} i.e $(\mathcal{N}, \mathcal{A}_j, \mathbb{E}_{\mathbf{S}}r_j(\mathbf{S},.)).$ A (state-independent) Nash equilibrium point \cite{nash1950} of the above robust game is a strategy profile such that no node can improve its payoff by  unilateral deviation, see Definition \ref{definition1} and Definition \ref{definition2}.

Since the current state is not observed by the nodes, it will be difficult to implement state-dependent strategy. Our goal is to design a learning algorithm for a state-independent equilibrium given in Definition \ref{definition1}. In what follows we assume that we are in a setting where the above problem has at least one isolated state-independent equilibrium solution. More details on existence of equilibria can be found in Theorem 3 in \cite{tian}.

\subsection{Learning algorithm}
Suppose that each node $j$  is able to observe a numerical value $\tilde{r}_{j,k}$ of the function $r_{j}(\mathbf{S}_k,\mathbf{a}_k)$ at time $k$, where $\mathbf{a}_k=(a_{j,k},\mathbf{a}_{-j,k})$ is the action  of node $j$ at time $k$. $\hat{a}_{j,k}$ is an intermediary variable. $a_j$, $\Omega_j$ $\phi_j$ represent the amplitude frequency and phase of the sinus perturbation signal given by $b_j\sin(\Omega_j \hat{k}+\phi_j)$, $\tilde{r}_{j,k+1}$ represents the payoff at time $k+1$. The learning algorithm is presented in Algorithm \ref{algorithm1} and is explained below. At each time instant $k$, each node updates its action $a_{j,k}$, by adding the sinus perturbation i.e. $b_j\sin(\Omega_j \hat{k}+\phi_j)$ to the intermediary variable $\hat{a}_{j,k}$ using equation (\ref{eq_Discrete_support}), and makes the action using   $a_{j,k}$. Then, each node  gets a realization of the payoff $\tilde{r}_{j,k+1}$ from the dynamic environment at time $k+1$ which is used to compute $ \hat{a}_{j,k+1}$ using equation (\ref{eq_Discrete_Main}). The action  $a_{j,k+1}$ is then updated using equation (\ref{eq_Discrete_support}). This procedure is repeated for the window $T$.

The algorithm is in discrete time and is given by
\begin{eqnarray}
    a_{j,k}&=&\hat{a}_{j,k}+b_j\sin(\Omega_j \hat{k}+\phi_j)\label{eq_Discrete_support} \\
    \hat{a}_{j,k+1}&=&\hat{a}_{j,k}+ \lambda_k z_j b_j \sin(\Omega_j \hat{k}+\phi_j)\tilde{r}_{j,k+1}\label{eq_Discrete_Main}
\end{eqnarray}

where $\hat{k}:=\sum_{k'=1}^k \lambda_{k'},$ $\Omega_j\neq \Omega_{j'}, \Omega_{j'}+\Omega_j\neq  \Omega_{j''}$   $\forall j,j',j''.$

For almost sure convergence, it is usual to consider vanishing step-size or learning rate such as $\lambda_k=\frac{1}{k+1}.$ However, constant learning rate $\lambda_k=\lambda$ could be more appropriate in some regime. The parameter $\phi_j$ belongs to $[0,2\pi] \forall \; j$, $k\in \mathbb{Z}_{+}$

\begin{algorithm}
\caption{Distributed learning algorithm}
\label{algorithm1}
\begin{algorithmic}[1]
\STATE Each node $j$, initialize $\hat{a}_{j,0}$ and transmit
\STATE Repeat
\STATE Calculate action $a_{j,k}$ according to Equation (\ref{eq_Discrete_support})
\STATE Perform action $a_{j,k}$
\STATE Observe $\tilde{r}_{j,k}$
\STATE Update $\hat{a}_{j,k+1}$ using Equation (\ref{eq_Discrete_Main})
%  \STATE Update the action $a_{i,k+1}$
\STATE until  horizon $T$
\end{algorithmic}
\end{algorithm}

\begin{rem}[Learning Scheme in Discrete Time]\label{remark2}
    As we will prove in subsection \ref{Convergence_to_ODE}, the difference equation (\ref{eq_Discrete_support})  can be seen as a discretized version of the learning scheme presented in \cite{Kristic}. But it is for games with state-dependent payoff functions i.e., robust games.
\end{rem}
It should be mentioned here for clarity that the action $a_{j,k}$ of each node $j$ is scalar.

\subsection{Interpretation of the proposed algorithm}

In some sense our algorithm is trying to estimate the gradient of the function $r_{j}(.)$, but we don't have access to the function but just its numerical value. The following equation clearly illustrated the significance of each variable and constant in the algorithm.

\begin{eqnarray}
\highlightedBox{\small\underbrace{\hat{a}_{j,k+1}}_{\substack{\text{New}\\ \text{Value}}} =\underbrace{\hat{a}_{j,k}}_{\substack{\text{Old}\\ \text{Value}}} + \overbrace{\lambda_k}^{\substack{\text{Learning}\\ \text{Rate}}} \underbrace{z_j}_{\substack{\text{Growth}\\ \text{Rate}}} \underbrace{b_j}_{\substack{\text{Perturbation}\\ \text{Amplitude}}} \sin(\overbrace{\Omega_j}^{\substack{\text{Perturbation}\\ \text{Frequency}}}  \hat{k}+\underbrace{\phi_j}_{\substack{\text{Perturbation}\\ \text{Phase}}} )\overbrace{\tilde{r}_{j,k+1}}^{\substack{\text{New}\\ \text{Reward}}}}
\end{eqnarray}

The learning rate $\lambda_k$ can be constant or variable depending on the requirement for the algorithm and system limitations. Perturbation amplitude $b_j>0$ is a small number. $z_j>$ is also a small value which can be varied for fine tuning. Rewriting the above equation we get

\begin{eqnarray}
\frac{\hat{a}_{j,k+1}-\hat{a}_{j,k}}{\lambda_k}&=&  z_j b_j \sin(\Omega_j \hat{k}+\phi_j)\tilde{r}_{j,k+1}
\end{eqnarray}

For vanishing step size as $k\longrightarrow\infty$ $\lambda_k\longrightarrow0$  and the trajectory of the above algorithm coincides with the  trajectory of the ODE in equation (\ref{eq_Stochastic_Main})

\section{Main results}\label{Mainresults}
In this section we present the convergence results as introduced in the contribution section.

We introduce the following assumptions that will be used step by step\footnote{We do not use A1 and A2 simultaneously.}.
\begin{assu}[A\ref{Assumption1}: Vanishing learning rate]\label{Assumption1}
%\noindent \textbf{Assumption A1 (vanishing learning rate)}:\\
$\lambda_k>0,\ \sum_k\lambda_k=\infty$,  $\sum_k|\lambda_k|^2<\infty$. There exists $C_0>0$ such that
$\mathbb{P}\left( \sup_{k} \parallel a_k\parallel <C_0\right)=1.$
The reason for A1 is that $\lambda_k$ represents the step size of the algorithm. So the sum over all $\sum_k\lambda_k=\infty$ as it needs to traverse over all discrete time. The condition $\sum_k|\lambda_k|^2<\infty$ ensures  bound for the cumulative noise error. This last assumption is for a local stability analysis.
\end{assu}
\begin{assu}[A\ref{Assumption2}: Constant learning rate]\label{Assumption2}
%\noindent \textbf{Assumption A2 (constant learning rate)}:
$\lambda_t=\lambda>0,$  $\sup_t[\mathbb{E}\|\mathbf{a}_t\|^2]^{\frac{1}{2}}<+\infty $ and  $\|\mathbf{a}_t\|^2$ is uniformly integrable.
\end{assu}
\begin{assu}[A\ref{Assumption3}:Existence of a local maximizer]\label{Assumption3}
%\noindent \textbf{Assumption A3 (Existence of a local maximizer):} \\
$  \mathbb{E}_\mathbf{S}\frac{\partial r_j(\mathbf{S},\mathbf{a}^*)}{\partial a_j} =0,\  \;    \mathbb{E}_{\mathbf{S}}\frac{\partial^2 r_j(\mathbf{S},\mathbf{a}^*)}{\partial a^2_j} <0.$
These  two conditions tell us that $a^*_j$ is a local maximizer of $a_j \longrightarrow \mathbb{E}_{\mathbf{S}}r_j(\mathbf{S},a_j,\mathbf{a}_{-j}^*)$ where $\mathbf{a}_{-j}^*=(a_{1}^*,\ldots,a_{j-1}^*,a_{j+1}^*,\ldots,a_{t}^*).$
\end{assu}
\begin{assu}[A\ref{Assumption4}: Diagonal Dominance]\label{Assumption4}
%\noindent \textbf{Assumption A4 (Diagonal Dominance):}
the expected payoff has a Hessian that is diagonally dominant at $\mathbf{a}^*$, i.e., {$\left|\mathbb{E}_{\mathbf{S}}\left(\frac{\partial^2r_j(\mathbf{S},\mathbf{a}^*)}{\partial a_j^2}\right)\right|- \sum_{j'\neq j} \left|\mathbb{E}_{\mathbf{S}}\left(\frac{\partial^2r_j(\mathbf{S},\mathbf{a}^*)}{\partial a_j\partial a_{j'}}\right)\right|>0.$}
Note that A4 implies that the Hessian of the expected payoff is invertible at $a^*.$ This assumption is weaker compared to the classical extremum seeking algorithm because the Hessian of $r_j(\mathbf{S},\mathbf{a}^*)$ does not need to be invertible for each $\mathbf{S}.$
%For example, in the Shannon payoff case, the Hessian is not invertible for $S=0.$

We assume  $\mathbf{S}\longmapsto r_j(\mathbf{S},\mathbf{a}) $ is integrable with respect to $\mathbf{S}$ so that the expectation
$\mathbb{E}_{\mathbf{S}}r_j(\mathbf{S},\mathbf{a})$ is finite.
\end{assu}

\subsection{Convergence to ODE} \label{Convergence_to_ODE}
\subsubsection*{Stochastic approximation}
First we need to show that our proposed algorithm converges to the respective ODE almost surely.
We will use a dynamical system viewpoint and stochastic approximation method to analyze our learning algorithm. The idea consists of finding the asymptotic pseudo-trajectory of the algorithm via ordinary differential equation (ODE). To do so, we use the framework initiated by Robbins-Monro\cite{monro1} or \cite{kiefer}. See \cite{benaim,borkar} for recent development. The works in \cite{benaim,borkar}  allows us to find the limiting trajectory of the learning algorithm.

Our scheme can be written as $ \hat{a}_{j,k+1}=\hat{a}_{j,k}+ \lambda_k z_j b_j \sin(\Omega_j \hat{k}+\phi_j)\tilde{r}_{j,k+1}.$
Now we rewrite the above equation in Robbins-Monro \cite{monro1} form as:
$ \hat{a}_{j,k+1}=\hat{a}_{j,k}+\lambda_k\left[ f_j(k,\mathbf{a}_k)+M_{k+1} \right], $
where $f_{j}(k,\mathbf{a}_{k})\triangleq z_j b_j \sin(\Omega_j \hat{k}+\phi_j)\mathbb{E}_{\mathbf{S}}r_{j}( \mathbf{S}, \mathbf{a}_k),$\\$ M_{k+1}\triangleq z_j b_j \sin(\Omega_j \hat{k}+\phi_j)\left[ \tilde{r}_{j,k+1}-\mathbb{E}_{\mathbf{S}}r_{j}( \mathbf{S}, \mathbf{a}_k)\right].$
%\begin{eqnarray*}
%f_{j}(k,\mathbf{a}_{k})&\triangleq &z_j b_j \sin(\Omega_j \hat{k}+\phi_j)\mathbb{E}_{\mathbf{S}}r_{j}( \mathbf{S}, \mathbf{a}_k)\\
%M_{k+1}&\triangleq &z_j b_j \sin(\Omega_j \hat{k}+\phi_j)\left[ \tilde{r}_{j,k+1}-\mathbb{E}_{\mathbf{S}}r_{j}( \mathbf{S}, \mathbf{a}_k)\right]
%\end{eqnarray*}

Since our payoff $r_{j}$ is Lebesgue integrable with respect to $\mathbf{S},$ expectation of payoff function {$\mathbb{E}_{\mathbf{S}}r_{j}( \mathbf{S}, \mathbf{a})$} is finite.  $M_{k+1}$ is clearly a martingale adapted to the filtration $\mathcal{F}_k$ generated by the random variable $\mathbf{S}_{k'}, k'\leq k$ and the
initial law of $\mathbf{a}_0.$ Moreover $M_{k+1}$ has a zero mean. Thus, $M_{k+1}$ is a difference martingale.

\begin{thm}[Variable Learning Rate] \label{Theorem1}
Under Assumption A1, the learning algorithm converges almost surely to the trajectory of a non-autonomous system given by
{%
\begin{eqnarray*}
\frac{d}{dt}\hat{a}_{j,t}&=& z_j b_j \sin(\Omega_j t+\phi_j)\mathbb{E}_{{\mathbf{S}}}\left(r_{j}(\mathbf{S},\mathbf{a}_t)\right)\\
a_{j,t}&=&\hat{a}_{j,t}+b_j \sin(\Omega_j t+\phi_j)
\end{eqnarray*}
}
\end{thm}

The gap between the interpolated version of algorithm and the solution of the ODE is bounded by\\
{
\begin{eqnarray*} \sup_{t\in[t_k, t_k+T]} \|\bar{\mathbf{a}}(t)-\mathbf{a}^{t_k}(t)  \| \leq K_{T,t}e^{LT}+C_T\lambda_{t+k} \end{eqnarray*}} which vanishes, where $\bar{\mathbf{a}}(t)$ is the interpolated version of the algorithm and $\mathbf{a}^{t_k}(t)$ is the solution of the ODE at time $t$ starting from $t_k:=\sum_{t'=0}^{k} \lambda_{t'},$ where $L$ is the Lipschitz constant for the ODE and $T$ is the time window. $K_{T,t}$ is specified below.

In order to calculate the bound we need to define a few terms which are helpful in obtaining a compact form of the bound.
{
  \begin{eqnarray}
  K_{T,t}&\triangleq&C_T L \sum_{k\geq0}\lambda_{t+k}^2+\sup_{k\geq0}\|\delta_{t,t+k}\|\\
  \delta_{t,t+k}&\triangleq&\xi_{t+k}-\xi_{t}\\
  \xi_{t}&\triangleq&\sum_{m=0}^{t-1}\lambda_m M_{m+1}\\
  C_T&\triangleq&\|r(0)\|+L(C_0+\|r(0)\|T)e^{LT}<\infty\\ \nonumber
    \end{eqnarray}
}

To prove that the learning algorithm (discrete ODE) converges to the ODE we need to verify conditions from Borkar \cite{borkar} Chapter 2 Lemma 1.
\begin{eqnarray*}
    \lim_{t\longrightarrow\infty}\sup_{s\in[t,t+T]}\|\tilde{a}_s-a_s^*\|&=&0 \; a.s.
\end{eqnarray*}

%{\bf How  important is this result?:}

This is an important result as it gives us an approximation on the error between our algorithm and the corresponding ODE.

\begin{thm}[Fixed Learning Rate] \label{Theorem2}
Under Assumption A2, the learning algorithm converges in distribution when $\lambda \longrightarrow 0,$ to the trajectory of a non-autonomous system given by
\begin{eqnarray}
    \frac{d}{dt}\hat{a}_{j,t}&=& z_j b_j \sin(\Omega_j t+\phi_j)\mathbb{E}_{\mathbf{S}}\left(r_j( \mathbf{S}, {\mathbf{a}}_t)\right)\\
    a_{j,t}&=&\hat{a}_{j,t}+b_j \sin(\Omega_j t+\phi_j)
\end{eqnarray}
    Moreover the error gap is in order of $\lambda$. As $\lambda$ converges to zero, the algorithm converges (in distribution) to the ODE.
\end{thm}

The advantage of Theorem \ref{Theorem2} compared to Theorem \ref{Theorem1} is the convergence time. The number of iterations required to reach a fixed time $T$ is less with constant learning rate than the vanishing learning rate. However, the convergence notion under constant step size is weaker (it is in distribution) compared to the almost surely convergence with vanishing learning rate. So there is a sort of tradeoff between almost sure convergence and convergence time.

Let $\Delta_t$ be the gap between the ODE and the isolated equilibrium at time $t.$

\begin{thm} [Exponential Stability]\label{Theorem3}
    Assume A3-A4 and Remark \ref{remark4},\ref{remark5} holds. Then,
    there exist $\acute{M}, \acute{m} > 0$ and
    $\bar \epsilon, \bar{b}_j$ such that, for all $\epsilon \in (0,\bar{\epsilon})$ and  $b_j \in (0,\bar{b}_j)$, if the initial gap is $\Delta_0$ (which is small)
    then for all time $t,$
    \begin{eqnarray} \label{ineqkrstic}
   \Delta_t \leq  y_{1,t}
 \end{eqnarray}
where
\begin{eqnarray} \label{ineqkrstic2}
        y_{1,t}& \triangleq &\acute{M}e^{-\acute{m} t} \Delta_0+O(\epsilon+\max_j  b_j^3)
\end{eqnarray}
\end{thm}

\begin{proof} [Sketch of Proof of Theorem \ref{Theorem3}]\label{proofthm3}

    Local stability proof of Theorem \ref{Theorem3} follows the steps in \cite{ShuJunLiu2011}.

%This completes the proof.
\end{proof}

From the above equation it is clear that as time goes to infinity the first term in $y_{1,t}$  bound vanishes exponentially and the error is bounded by the amplitude of the sinus perturbation i.e. $O(\epsilon+\max_j  b_j^3)$. This means that the solution of ODE converges locally exponentially to the state-independent equilibrium action $\mathbf{a}^*$ provided the initial solution is relatively close.

\begin{defi}[$\epsilon-$Nash equilibrium payoff point]\label{definition3}
An  $\epsilon-$Nash equilibrium point in state-independent strategy is a strategy profile such that no node can improve its payoff more than $\epsilon$ by  unilateral deviation.
\end{defi}

\begin{defi}[$\epsilon-$close Nash equilibrium strategy point]\label{definition4}
An  $\epsilon-$close Nash equilibrium point in state-independent strategy is a strategy profile such that the Euclidean distance to a Nash equilibrium is less than $\epsilon.$
\end{defi}
A $\epsilon-$close Nash equilibrium point is an approximate Nash point with a precision at most $\epsilon.$

It is not difficult to see that for Lipschitz continuous payoff functions, an  $\epsilon-$close Nash equilibrium is an $L\epsilon-$Nash equilibrium point where $L$ is the Lipschitz constant.

Next corollary shows that one can get an $\epsilon-$close Nash equilibrium in finite time.

\begin{coro} [Convergence Time] \label{Corollary1}
    Assume  A3-A4 and Remark \ref{remark4},\ref{remark5} holds. Then, the ODE reaches a
    $(2\epsilon+\max_j b_j^3)-$close to a Nash equilibrium in at most $T$ time units where
    $        T =   \frac{1}{\acute{m}} \log(\frac{\Delta_0 \acute{M}}{\epsilon}) $
\end{coro}

\begin{proof}[Sketch of Proof for Corollary \ref{Corollary1}]
    The proof follows from the inequality (\ref{ineqkrstic}) in Theorem \ref{Theorem3}.
\end{proof}
%{\bf Comment the result:}

\begin{coro}[Convergence to the ODE] \label{Corollary2}
    Under Assumption A1, A3, and A4, the following inequality holds almost surely:
    $  \parallel \tilde{\mathbf{a}}_t -\mathbf{a}^*\parallel \leq  y_{1,t}+y_{2,t}  $
\end{coro}
where
\begin{eqnarray} \label{ineqBorkar}
y_{2,t} \triangleq   C_T (\lambda_{t+k}+L\sum_{k'\geq0}\lambda_{t+k'}^2 )+\sup_{k'\geq0}\|\delta_{t,t+k'}\|
\end{eqnarray}

\begin{proof}[Proof of Corollary \ref{Corollary2}]
The proof uses the triangle inequality  $\parallel \tilde{\mathbf{a}}_t -\mathbf{a}^*\parallel \leq \parallel\tilde{\mathbf{a}}_t -\mathbf{a}_t\parallel +\parallel {\mathbf{a}}_t -\mathbf{a}^*\parallel.$ By Theorem \ref{Theorem1}, one gets $ \parallel\tilde{\mathbf{a}}_t -\mathbf{a}_t\parallel \leq y_{1,t}$ and by Theorem \ref{Theorem3}, one has $\parallel {\mathbf{a}}_t -\mathbf{a}^*\parallel \leq y_{2,t}$
Combining together, one arrives at the announced result.
\end{proof}

%\section{Parameter design}

Then constants in equation (\ref{ineqkrstic2}) and (\ref{ineqBorkar}) depends on the number of players and the dimension of the action space.

%The above equation needs to be invertible to get the value of $\lambda.$

%{\bf \color{red} To be completed}
\subsection{Convergence of the stochastic ODE}

In this subsection we study the stochastic ODE given by

\begin{eqnarray}
    a_{j,t}&=&\hat{a}_{j,t}+b_j\sin(\Omega_j t+\phi_j)\label{eq_Stochastic_support}\\
    \frac{d}{dt}\hat{a}_{j,t}&=&  z_j b_j \sin(\Omega_j t+\phi_j)\tilde{r}_{j,t}\label{eq_Stochastic_Main}
\end{eqnarray}

where $r_{j,t}$ is the realization of the state-dependent payoff $r_j(\mathbf{S}_t,\mathbf{a}_t)$ at time $t.$ We assume the state process is ergodic so that,
$$
\lim_{T\longrightarrow \infty}\frac{1}{T}\int_0^T   \mu_{j}(t)r_j(\mathbf{S}_t,\mathbf{a}_t) \ dt=\lim_{T\longrightarrow \infty}\frac{1}{T}\int_0^T   \mu_{j}(t)\mathbb{E}_\mathbf{S}r_j(\mathbf{S},\mathbf{a}_t) \ dt
%\end{equation}
$$
In particular the asymptotic drift of the deterministic ODE and the stochastic ODE are the same. Hence, the following theorem follows:

\begin{thm}[Almost sure exponential stability]\label{Theorem4}
The stochastic algorithm  \ref{algorithm1} converges asymptotically almost surely to the stochastic ODE in equation (\ref{eq_Stochastic_Main})
i.e.
$$P( \parallel \tilde{\mathbf{a}}_t -\mathbf{a}^*\parallel \leq  y_{1,t}+y_{2,t} )=1\; a.s.$$

\end{thm}

Since the state process is ergodic, we can apply the stochastic averaging theorem from \cite{stochasticKrstic} to get the announced result.
\section{Numerical Example: A Generic Wireless Network with Interference\label{NumericalExample}}

Even though the distributed optimization problem, considered in this paper, and the developed approach are general and can be used in many application domains. As an application of the above framework, we will consider the problem of  power control in wireless networks in order to better illustrate our contribution.  Consider an interference channel composed of  $N$ transmit receiver pairs as shown in Figure \ref{Dynamic_Enviornment}. Each transmitter communicates with its corresponding receiver and incurs an interference on the other receivers. Each receiver feeds back a numerical value of the payoff $\tilde{\gamma}_{j} (\mathbf{H},\mathbf{p})$ to its corresponding transmitter.

The problem is composed of transmitter-receiver pairs; all of them use the same frequency and thus generate interference onto each other. Each transmitter-receiver pair  has therefore its own payoff/reward/utility function that depends necessarily on the interference exerted by the other pairs/nodes.  Since the wireless channel is time varying as well as the interference, the objective is necessarily to optimize in the long-run (e.g. average) the payoff functions of all the nodes. The payoff function of node $j$ at time $k$ is denoted by  $r_{j}(\mathbf{H}_k,\mathbf{p}_k)$ where $\mathbf{H}_k:=[h_{k}(i,j)]$ represents an $N\times N$ matrix containing  channel coefficients at time $k$, $h_{k}(i,j)$ represents the channel coefficient between transmitter $i$ and receiver $j$ (where $(i,j)\in\mathcal{N}^2$) and $\mathbf{p}_k$ represents the vector containing transmit powers of $N$ transmit-receive nodes. The most common technique used to obtain a local maximum of the nodes' payoff functions is the gradient based descent or ascent method.

\begin{rem}\label{remark3}
%    In the following algorithm we use $a_{j,k}$ to represent the action, but as this is a general algorithm it can be use to solve other types of problems.
    \begin{table}[htb]
\caption{Equivalent Notations for Wireless} \label{table_of_notations2}
\begin{center}
\begin{tabular}{c|c|c}
\hline
  General &  Application & Description \\ \hline
    $\tilde{r}_{j,k}$ &$ \tilde{\gamma}_{j,k}$ & utility/payoff of transmitter $j$ at time $k$\\
    $a_{j,k}$ & $p_{j,k}$ & action/power of transmitter $j$ at time $k$\\
 $s_{jj',k}$ & $g_{jj',k}$ & state/channel gain between transmitter \\
 && $j$  and  receiver $j'$ at time $k$\\
\hline
\end{tabular}
\end{center}
\end{table}
\end{rem}

    \begin{figure*}[!hbtp]
        \begin{center}
            \tikzstyle{block_pu} = [draw,dashed, draw=black!90,thick,rectangle, minimum height=6em, minimum width=18em,roundedcorners=.5em]

\tikzstyle{block_su} = [draw,dashed, draw=black!90!black,thick,rectangle, minimum height=6em, minimum width=18em,roundedcorners=.5em]

\tikzstyle{antenna} = [draw, draw=black!50,thick,fill=black!20, isosceles triangle,rotate=-90,scale=.4 ]

\tikzstyle{sum2}=[ node distance=2cm]

%\tikzstyle{antenna} = [draw, draw=black!50,thick,fill=black!10,
%triangle, minimum height=6em, minimum width=3em]
\tikzstyle{sum}=[draw,draw=black!90,thick,double, fill=black!10, circle , node distance=2cm]
%\tikzstyle{input} = [coordinate]
%\tikzstyle{output} = [coordinate] \tikzstyle{pinstyle} = [pin edge={to-,thin,black}]

%\pgfdeclarelayer{background layer} \pgfdeclarelayer{foreground
%layer} \pgfsetlayers{background layer,main,foreground layer}

\begin{tikzpicture}%[auto, node distance=6cm,>=latex']

\foreach \r/\s/\a  in {1/1/1,2/j/0,3/N/1}
%\foreach \r/\s/\a  in {1/1/1,2/2/0,3/j/1,4/4/0,5/N/1}
{
    \ifnum \a=0
        \node [sum2] (r\r) at (6.5em,-6*\r em) {$\vdots$};
    \else
        \node [sum] (r\r) at (6.5em,-6*\r em) {$Rx_\s$};
    \fi
    \foreach \t/\u/\b in {1/1/1,2/j/0,3/N/1}
%    \foreach \t/\u/\b in {1/1/1,2/2/0,3/j/1,4/4/0,5/N/1}
    {
        \ifnum \b=0
            \node [sum2] (t\t) at (-6.5em,-6*\t em) {$\vdots$};
        \else
            \node [sum] (t\t) at (-6.5em,-6*\t em) {$Tx_\u$};
%            \node [sum2] (lt\t) [left=of t\t] {p_{\t,k}};
%            \node [left of] {$g_{F_1f_1}$} (t\t);
%            node[below,sloped] {$g_{F_1f_1}$}
        \fi

    \ifnum \r=\t  % Drawing the lines

       \ifnum \r=0
%            \node [antenna] (t\t) at (-6.5em,-6*\t em) {$Tx_\t$};
        \else
             \ifnum \r=4

                \draw [black!90, -stealth ,shorten >=.25em,thick,bend right=45,looseness=1] (t\t.east) -- node [above] {$g_{\u\s}$} (r\r.west) ;
                \draw [dashed, color=black!50!black, -stealth,shorten >=.25em,thick, bend right, looseness=1] (r\r.north) to node [below] {$\tilde{\gamma}_\s$} (t\t.north);
            \else
                \draw [black!90, -stealth ,shorten >=.25em,thick,bend right=45,looseness=1] (t\t.east) -- node [above] {$g_{\u\s}$} (r\r.west) ;
                \draw [dashed, color=black!50!black, -stealth,shorten >=.25em,thick, bend right, looseness=1] (r\r.north) to node [below] {$\tilde{\gamma}_\s$} (t\t.north);
            \fi
        \fi
    \else
        \ifnum \r=3
%            \node [antenna]  at (-6.5em,-6*\t em) {$Tx_\t$};
            \draw [black!50, -stealth ,shorten >=.25em,thick,bend right=45,looseness=1.2] (t\t.east) -- node [near end,sloped,below] {$g_{\u\s}$} (r\r.west);
        \else
            \ifnum \r=4
    %            \node [antenna]  at (-6.5em,-6*\t em) {$Tx_\t$};
                \draw [black!50, -stealth ,shorten >=.25em,thick,bend right=45,looseness=1.2] (t\t.east) -- node [near end,sloped,below] {$g_{\u\s}$} (r\r.west);
            \else
                \draw [black!50, -stealth ,shorten >=.25em,thick,bend right=45,looseness=1.2] (t\t.east) -- node [near end,sloped,below] {$g_{\u\s}$} (r\r.west);
            \fi
         \fi
    \fi
    }
}
%    \node [block_pu, name=pu] at (0,-5em) {};
%    \node [above,black!50] at (pu.north) {Transmitter Receiver 1};
%
%    \node [block_su] (su) at (0,-13em) {};
%    \node [below,black!50!black]at (su.south) {Transmitter Receiver 2};
\end{tikzpicture}
            \caption[Interference Channel Model]{\label{Dynamic_Enviornment}Interference Channel Model}
        \end{center}
    \end{figure*}
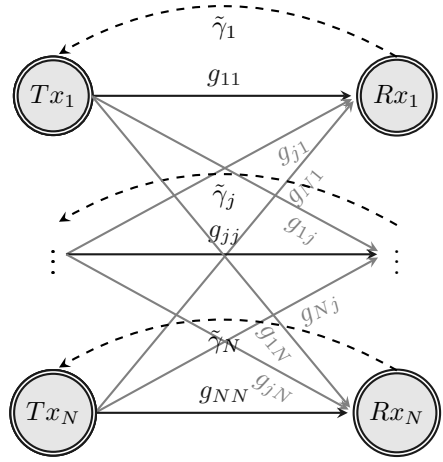

In section \ref{Mainresults}, we proved that our proposed algorithm converges to  $p^*$ for any type of payoff functions which satisfies the assumptions in section \ref{Convergence_to_ODE}. In order to show numerically that our algorithm converges to $p^*$, we run our algorithm for a simple payoff function. In parallel, we obtain analytically the Nash equilibrium $p^*$ and compare the convergence point of our algorithm to $p^*$. We therefore choose a simple payoff function for which $p^*$ can be obtained analytically.
%(in order to make the comparison with our algorithm, see the explicit expression in Appendix ).   % simple payoff function given as,

The payoff function of node $j$ at time $k$ has then the following form:

\begin{eqnarray}\nonumber
\highlightedBox{\tilde{\gamma}_{j}(\mathbf{H}_k,\mathbf{p}_k)=\underbrace{\omega}_{\text{bandwidth}}\underbrace{\log(1+\frac{p_{j,k}g_{jj,k}}{\sigma^2+\sum_{j'\neq j} p_{j',k}g_{j'j,k}})}_{\text{Rate}}-\underbrace{\kappa p_{j,k}}_{\text{constraint on powers}}
}
\end{eqnarray}
where $\omega$ represents the bandwidth available for transmission. The above payoff function $\tilde{\gamma}_{j}(\mathbf{H}_k,\mathbf{p}_k)$ consists of $log$ of $(1+SINR)$ of  user $j$ and the unit cost of transmission is $\kappa$. It is assumed that a used doesn't know the structure function $\tilde{\gamma}_{j}(.)$ or the law of the channel state.
For the above payoff function to ensure the assumption A3-A4 and Remark \ref{remark4},\ref{remark5} we need to satisfy the condition { $\mathbb{E}|h_{jj}|^2\geq \mathbb{E}\sum_{j'\neq j} |h_{j'j}|^2$.} Please see appendix for more details.

The problem here is to maximize the payoff function $\tilde{\gamma}_j(\mathbf{H},\mathbf{p})$ which is stated as follows: find $\mathbf{p}^*$ such that for each user $j\in\mathcal{N},$ satisfies \\
$
p_j^*\in \arg\max_{p_j\geq 0}\mathbb{E}\tilde{\gamma}_{j}(\mathbf{H},p_1^*,\ldots, p_{j-1}^*,p_j,p_{j+1}^*,\ldots,p^*_N).
$
Note that when $g_{jj}=0$ then the payoff of user $j$ is negative and  the minimum power ${p}^*_j=0$ is a solution to the above problem. For the remaining, we assume that $|h_{jj}|^2=g_{jj}> 0.$

The  channel $h_{j,j'}$ is time varying and is generated using an independent and identically distributed complex gaussian channel model with variance $\sigma_{jj'}^2$ such that $\sigma_{jj}= 1$  $\sigma_{jj'}= 0.1, j'\neq j                                                                                           $.  The thermal noise is assumed to be a zero mean gaussian with variance $\sigma^2$ such that $\sigma^2=1.$

We consider the following simulation settings with $N=2$ for the above wireless model:
$k_1=0.9,k_2=0.9,\phi_1=0,\phi_2=0$,$\Omega_1=0.9,\Omega_2=1,$ $b_1=0.9,b_2=0.9.$ The numerical setting could be tuned in order to make the convergence slower or faster with some other tradeoff. Due to space limitations further discussion on how to select these parameters has been omitted. $p_{1,0}$ and $p_{2,0}$ represent the starting points of the algorithm which are initialized as $p_{1,0}=p^*_{1}+10$ and $p_{2,0}=p^*_{2}+10$. $\kappa=2$ is the penalty for interference, $\omega=10$ is the bandwidth and the variance of noise is normalized.
Figure \ref{power1yt04} represents the average transmit power trajectories of the algorithm for two nodes. The dotted line represents $p^*$.
As can be seen from the plots that the system converges to $p^*$ where $p_j^*= 3.9604, \;j\in \{1,2\}$.

\begin{figure*}[htb]
%  \begin{minipage}[htb]{6cm}
  \begin{center}
    \includegraphics[width=12cm]{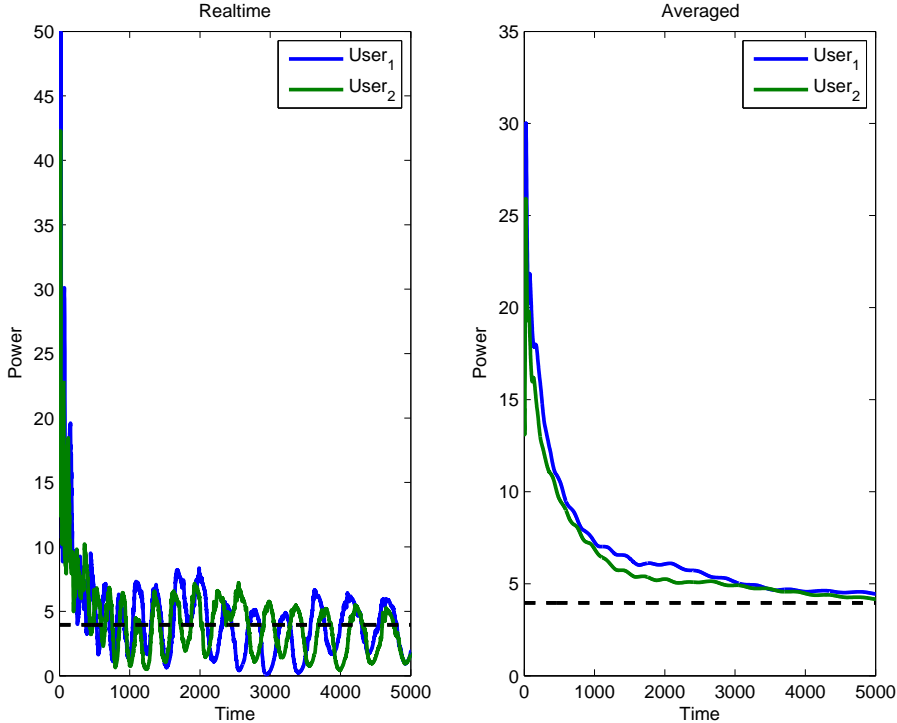}\\
  \caption{Power  evolution (discrete time)}\label{power1yt04}
    \end{center}
\end{figure*}

\begin{figure*}[htb]
%  \begin{minipage}[htb]{6cm}
  \begin{center}
    \includegraphics[width=12cm]{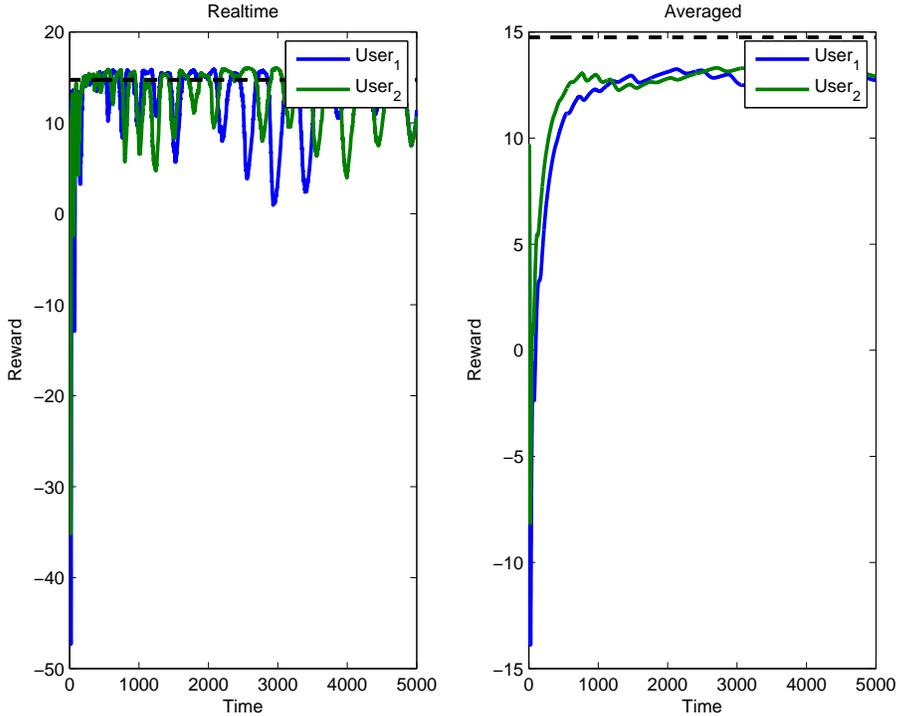}\\
  \caption{Payoff  evolution (discrete time)}\label{power1yt05}
    \end{center}
\end{figure*}

%\section{Discussions}\label{Discussions}

The example we discussed is only one of the possible types of applications where our proposed algorithm can be implemented.

%{\bf \color{red} To be completed}
Consider for example the following payoffs:  $q_1(.)=goodput(.)$ and $q_2(.)=\mathbb{P}(goodput(.)<\eta)$ where $\eta$  is a small value and $\mathbb{P}(.)$ stands for probability.
%\begin{itemize}
%  \item  $q_1(.)=goodput(.)$
%  \item  $q_2(.)=\mathbb{P}(goodput(.)<\eta)$ where $\eta$  is a small value and $\mathbb{P}(.)$ stands for probability.
%\end{itemize}
Goodput represents the ratio of correctly received information bits vs the number of transmitter bits. In wireless communications the  channel is constantly changing due to various physical phenomenon and interference from other sources and changes in the environment. It is hard to have a closed form expression for $q_1(.)$ due to complexity of the transmitter, receiver and unknown parameters. In practice, at each time $k$, the receiver has therefore a numerical value of $goodput(.)$ but no closed form expression for rate/goodput is available especially for advanced coding scheme (e.g. turbo code, etc.). $q_2(.)$ represents an outage probability for which also depends on the goodput, the gradient for $q_2(.)$ is notoriously hard to compute without channel and interference statistics knowledge (probability distribution function)  and closed form expression of $goodput(.)$. Our scheme can be particularly helpful in such scenarios.

The price/design parameter $\kappa$ inside the reward function can be tuned such that the solution of the distributed robust extremum coincides with a global optimizer of the system designer. The $\kappa$ can be same for all nodes or each node can have its own $\kappa_j$. Let $\mathbf{a}_g^*$ represent the optimal action or set of actions to be performed by each node to maximize their respective utilities.  It is possible to set $\kappa$ such that the following equation is satisfied.$\mathbf{a}(\kappa)=\mathbf{a}_g^*.$
$\kappa$ could represent a scalar or a vector depending on the system size and the application. To be able to effectively make  $\mathbf{a}(\kappa)$ equal to $\mathbf{a}^*$ we need to have enough degrees of freedom in the system. However this type of tuning is not true in general.

\section{Concluding remarks}\label{Conclusion}
\begin{description}
    \item[Work Presented:] In this paper we have presented a Nash seeking algorithm which is able to find the local minima using just the numerical value of the stochastic state dependent payoff function at each discrete time sample. We proved the convergence of our algorithm to a limiting ODE. We have provided as well the error bound for the algorithm and the convergence time to be in a close neighborhood of the Nash equilibrium.  A numerical example for a generic wireless network is provided for illustration. The convergence bounds achieved by our method are dependent on the step size and the perturbation amplitude.
  \item[New Class of Functions:] In this work we introduced a new class of state dependent payoff functions  $r_j(\mathbf{S},\mathbf{a})$ which are inspired from wireless systems applications. But these kind of functions are more general and appear in other application areas.
  \item[Achievable Bounds:] As it is clear from results in Theorem \ref{Theorem1} that convergence depends on an exponential term and the amplitude of the sinus perturbation. As amplitude becomes smaller, the error bound also vanishes. In contrast the standard stochastic subgradient method only depend on the step size.
  \item[Global Analysis:] All the work considered in this paper including Krstic et.al. consider local stability. Our work is an extension of their work and works for local stability.
       %In this work we deal with local stability of Nash equilibrium,
       The future work will focus on the extension to the case of Global Stability of Nash equilibrium for both deterministic and stochastic payoff functions.
  \item[Multidimensional Aspect:] The presented work has been studied for scalar reward and scalar action by each node. Scalar scenario has several applications to wireless (as in the aforementioned example) and sensor networks and numerous examples can be considered.  A possible extension to this work could be in the direction of vector actions where each users is able to perform multiple actions based on multiple rewards.
\end{description}

%{\bf \color{red} To be completed}
%\newpage
%
\appendix

\section{Convergence Theorems}\label{appendix_Convergence}

\subsection{Variable Step Size: Proof of Theorem \ref{Theorem1}}
The Theorem \ref{Theorem1} states that
Under Assumption A1, the learning algorithm converges almost surely to the trajectory of a non-autonomous system given by

  \begin{eqnarray*}
        \frac{d}{dt}\hat{{a}}_{j,t}&=& z_j b_j \sin(\Omega_j t+\phi_j)\mathbb{E}_{{\mathbf{S}}}\left(r_j( {\mathbf{S}}, \mathbf{a}_t)\right)\\
        {a}_{j,t}&=&\hat{{a}}_{j,t}+b_j \sin(\Omega_j t+\phi_j)
    \end{eqnarray*}

The proof follows in several steps.
\begin{itemize}
  \item The first step provides conditions for  Lipschitz continuity of the expected payoff which is given in Lemma \ref{lemma1}.
  From Lemma \ref{lemma2} we have that  $\forall j,t$, $f_j(t,\mathbf{a}) \triangleq b_j z_j\sin(\Omega_j t+\phi_j)\mathbb{E}_\mathbf{S}r_{j}(\mathbf{S},\mathbf{a}),$ is Lipschitz  over the domain $\mathcal{D}$
  \item Second step: the learning rates are chosen such that they satisfy assumption $A1.$
  \item Third step: we check the noise conditions.
\end{itemize}

\begin{lem} \label{lemma1}
Let
  \begin{eqnarray*}
    &&(\mathbf{S},\mathbf{a})\longmapsto r_j(\mathbf{S},\mathbf{a})
    \forall \mathbf{S} \in \mathcal{S}, \exists\; L_{j,\mathbf{S}}\ \mbox{such that} \\
    (C_1)& : &\|r_j(\mathbf{S},\mathbf{a})-r_j(\mathbf{S},\mathbf{a}')\|\leq L_{j,\mathbf{S}}\|\mathbf{a}-\mathbf{a}'\| \;\forall (\mathbf{a},\mathbf{a}')\in \mathcal{A}\\
    (C_2)& : &\mathbb{E}_\mathbf{S} L_{j,\mathbf{S}} <+\infty
  \end{eqnarray*}
    then the mapping $\mathbf{\mathbf{a}}\longmapsto \mathbb{E}_\mathbf{S} r_j(\mathbf{S},\mathbf{a})$ is Lipschitz  with Lipschitz  constant $L_j= \mathbb{E}_\mathbf{S}L_{j,\mathbf{S}}$
\end{lem}

\begin{proof} [Proof of Lemma \ref{Theorem1}]\label{prooflemma1}
Suppose that $\mathbf{a}\longmapsto E_\mathbf{S} r_j(\mathbf{S},\mathbf{a})$ is Lipschitz  with Lipschitz  constant $L_{j,\mathbf{S}}$, then by Jensen's inequality one has
  \begin{eqnarray*}
 \|\mathbb{E}_\mathbf{S} r_j(\mathbf{S},\mathbf{a})-\mathbb{E}_\mathbf{S} r_j(\mathbf{S},\mathbf{a}')\|& \leq &\mathbb{E}_\mathbf{S}\|r_j(\mathbf{S},\mathbf{a})-r_j(\mathbf{S},\mathbf{a}')\|
  \end{eqnarray*}

By condition $C_2,$ $\mathbb{E}_\mathbf{S} L_{j,\mathbf{S}} <+\infty.$ Let $L_j$ be  $\mathbb{E}_\mathbf{S}L_{j,\mathbf{S}}.$ Then

  \begin{eqnarray*}
 \|\mathbb{E}_\mathbf{S} r_j(\mathbf{S},\mathbf{a})-\mathbb{E}_\mathbf{S} r_j(\mathbf{S},\mathbf{a}')\|
                                                & \leq & L_{j} \|\mathbf{a}-\mathbf{a}'\|
  \end{eqnarray*}

  This completes the proof.

\end{proof}
\begin{rem}\label{remark4}

\begin{itemize}\item Note that under $C_1$ and $C_2$ the expected payoff vector $r=(r_j)_{j\in\mathcal{N}}$ is Lipschitz  continuous with $\tilde{L}=\max_{j} L_j, $

\item
If $\mathcal{S}$ is a compact set and $\mathbf{S}\longmapsto L_{j,\mathbf{S}}$ is continuous then $\mathbf{a}\longmapsto \mathbb{E}_\mathbf{S} r_j(\mathbf{S},\mathbf{a})$ is Lipschitz  [In particular, the condition $C_2$ is not needed]
\end{itemize}
\end{rem}
We shall prove the above remark by \textit{Reductio ad absurdum}. To prove the second statement of Remark \ref{remark4} we use compactness and continuity argument. We start from Bolzano--Wierstrass theorem which states that. For any $k,$ any continuous map $x\longmapsto f(k,\mathbf{a})$ over a compact set  $\mathcal{D}$ has at least one maximum, i.e.,
$\sup f(k,\mathbf{a})=\max_{\mathbf{a}\in \mathcal{D}} f(k,\mathbf{a})<\infty.$ The proof of this statement can be easily done by contradiction.  Suppose  $\sup f(k,\mathbf{a})=\infty.$ Then there exists a sequence $(a_l)_l$ such that  $ \mathbf{a}_l\in\mathcal{D}$ but $f(k,\mathbf{a}_l)\longrightarrow \infty$ as $l$ goes to infinity. This is impossible because $\mathcal{D}$ is compact which implies that $f(k,\mathcal{D})=\{ f(k,\mathbf{a})\ | \mathbf{a}\in\mathcal{D}\}$ is bounded by continuity.

Since $\mathcal{S}$ is compact and $\mathbf{S}\longmapsto L_\mathbf{S}$ is continuous, $ \sup_{\mathbf{S}\in\mathcal{S}} L_\mathbf{S}$ is also finite.

\begin{rem}\label{remark5}
If $r_j(\mathbf{S},\mathbf{a})$ is continuously differentiable with the respect to $\mathbf{a}$ then it is sufficient to check the expectation of the gradient is bounded (in norm).

if $\mathcal{S}$ is in Euclidean Space
\begin{itemize}
  \item $r_j$ is differentiable w.r.t $\mathbf{a}$
  \item $r_j(\mathbf{S},\mathbf{a})$, $\nabla_\mathbf{a}r_j(\mathbf{S},\mathbf{a})$ are continuous in $\mathbf{S}$
  \item $r_j(\mathbf{S},\mathbf{a})$, $\nabla_\mathbf{a}r_j(\mathbf{S},\mathbf{a})$ are absolutely integrable in $\mathbf{S}$ and $\mathbb{E}_\mathbf{S}r_j(\mathbf{S},\mathbf{a})$ is continuous in $\mathbf{a}.$
\end{itemize}
then
$$
\mathbb{E}[\nabla_\mathbf{a}r_j(\mathbf{S},\mathbf{a})]=\nabla_\mathbf{a}\mathbb{E}[r_j(\mathbf{S},\mathbf{a})]
$$
which can be written as
$$
\int_\mathbf{S}\nabla_\mathbf{a}r_j(\mathbf{S},\mathbf{a})\gamma(d\mathbf{S})=
\nabla_\mathbf{a}\int_\mathbf{S}r_j(\mathbf{S},\mathbf{a})\gamma(d\mathbf{S})
$$
where $\gamma$ is the measure of $\mathbf{S}$ state space. For more details on the above conditions please refer to  \cite{LaurentSchwartz}.

Since $f_j$ is a function of time and the actions of nodes, we need a uniform Lipschitz  condition on $f_j.$

We have
$$
|f_j(t,\mathbf{a})-f_j(t,\mathbf{a}')|\leq b_jz_j|\sin(\Omega_j t+\phi )|\left[  \|\mathbb{E}_\mathbf{S} r_j(\mathbf{S},\mathbf{a})-\mathbb{E}_\mathbf{S} r_j(\mathbf{S},\mathbf{a}')\| \right]
$$
 But one has $|\sin(.)|\leq 1.$ Hence,
 $$
|f_j(t,\mathbf{a})-f_j(t,\mathbf{a}')|\leq b_jz_j\left[  \|\mathbb{E}_\mathbf{S} r_j(\mathbf{S},\mathbf{a})-\mathbb{E}_\mathbf{S} r_j(\mathbf{S},\mathbf{a}')\| \right]
$$
  We use Lemma \ref{lemma1},

   $$
|f_j(t,\mathbf{a})-f_j(t,\mathbf{a}')|\leq  b_jz_j L_{j} \|\mathbf{a}-\mathbf{a}'\|
$$
This implies that the Lipschitz  constant of $f_j$ is less than the one of $r_j$ times the factor $b_jz_j.$

\end{rem}

Finally, we check  the noise conditions. The recursion equation is given by

$$
    {a}_{j,k+1}={a}_{j,k}+\lambda_k[f_{j}(k,{a}_{k})+M_{j,k+1}]
$$

where $M_{j,k+1}$ is a martingale difference sequence.  By definition the martingale sequence for the  algorithm is given as

$ M_{j,k+1}\triangleq z_j b_j \sin(\Omega_j\hat{k}+\phi_j)\left[ \tilde{r}_{j,k+1}-\mathbb{E}_{{\mathbf{S}}}[\tilde{r}_{j,k+1}( {\mathbf{S}}, \mathbf{a}_{k+1})]\right]$

which satisfied the condition $\mathbb{E}[M_{k+1}|\mathcal{F}_k]=0$  for $k\geq0$ almost surely (a.s.)

\begin{lem} \label{lemma2}
If $a_k\in\mathcal{D}$ then  the martingale is square-integrable with

   \begin{eqnarray*}
            \mathbb{E}[\|M_{k+1}\|^2|\mathcal{F}_k]%&\leq& \acute{c} \;\forall k\\
                                    &\leq& \acute{c}(1+\|\mathbf{a}_k\|^2) \;\forall k
    \end{eqnarray*}

\end{lem}

\begin{proof}[Proof of Lemma \ref{lemma2}]

Let $\tilde{r}_{j,k+1}$ be  the realization the payoff at time $k+1.$ The expected value of this random variable can be bounded above the norm of $\mathbf{a}_k.$

%$ M_{i,k+1}= l_i b_i \sin(\Omega_i\hat{k}+\phi_i)\left[ r_{i,k+1}-\mathbb{E}_{{\mathbf{S}}}[r_{i,k+1}( {\mathbf{S}}, \mathbf{a}_{k+1})]\right$

  \begin{eqnarray*}
 M_{j,k+1}&=& z_j b_j \sin(\Omega_j\hat{k}+\phi_j)( \tilde{r}_{j,k+1}-\mathbb{E}_{{\mathbf{S}}}[\tilde{r}_{j,k+1}( {\mathbf{S}}, \mathbf{a}_{k+1})])\\
\| M_{j,k+1}\|&\leq& |z_j| |b_j| |(\sin(\Omega_j\hat{k}+\phi_j)|\| \tilde{r}_{j,k+1}-\mathbb{E}_{{\mathbf{S}}}[\tilde{r}_{j,k+1}( {\mathbf{S}}, \mathbf{a}_{k+1})]\|)\\
                &\leq& z_j b_j (\| \tilde{r}_{j,k+1}\|+\|\mathbb{E}_{{\mathbf{S}}}[\tilde{r}_{j,k+1}( {\mathbf{S}}, \mathbf{a}_{k+1})]\|)\\
                &\leq& z_j b_j (\| \tilde{r}_{j,k+1}\|+\mathbb{E}_{{\mathbf{S}}}\|\tilde{r}_{j,k+1}( {\mathbf{S}}, \mathbf{a}_{k+1})\|)\\
                &\leq& z b   (\| \tilde{r}_{j,k+1}\|+\mathbb{E}_{{\mathbf{S}}}\|\tilde{r}_{j,k+1}( {\mathbf{S}}, \mathbf{a}_{k+1})\|)
   \end{eqnarray*}

   Where  $|\sin(.)|\leq 1, \;z\triangleq\max|z_j|, \; b\triangleq\max|b_j|, $
  $\| \tilde{r}_{j,k+1}\|$ is bounded because of the Lipschitz condition as mentioned in $C_1,$ which is shown below.

  \begin{eqnarray} \label{Q1}
  \|r_j(\mathbf{S},\mathbf{a}_k)-r_j(\mathbf{S},0)\|&\leq &L_{j,\mathbf{S}}\|\mathbf{a}_k-\mathbf{0}\| \;\forall (\mathbf{a}_k)\in \mathcal{A}\\\nonumber
\|r_j(\mathbf{S},\mathbf{a}_k)\|&\leq &\|r_j(\mathbf{S},0)\|+L_{j,\mathbf{S}}\|\mathbf{a}_k\| \\\nonumber
                                    &\leq &\beta_{1,\mathbf{S}}+L_{j,\mathbf{S}}\|\mathbf{a}_k\|
%                                &\leq &\|r_j(\mathbf{S},0)\|+L_{j,\mathbf{S}}\|\mathbf{a}\|
%                                &\leq &L_{j,\mathbf{S}}\|\mathbf{a}\|
\end{eqnarray}
Where $\beta_{1,\mathbf{S}}\triangleq\|r_j(\mathbf{S},0)\|.$
 The above equations \ref{Q1} show that $\|r_j(\mathbf{S},\mathbf{a})\|$ is bounded by $\beta_{1,\mathbf{S}}+L_{j,\mathbf{S}}\|\mathbf{a}\|$. By taking expectation of the above set of inequalities we get.
 \begin{eqnarray}\label{Q2}
\mathbb{E}_\mathbf{S} \|r_j(\mathbf{S},\mathbf{a}_k)\|
&\leq &\mathbb{E}_\mathbf{S}\| r_j(\mathbf{S},0)\|+\mathbb{E}_\mathbf{S} L_{j,\mathbf{S}}\|\mathbf{a}_k\|\\\nonumber
%&\leq & \mathbb{E}_\mathbf{S} \|r_j(\mathbf{S},0)\|+\mathbb{E}_\mathbf{S}  L_{j,\mathbf{S}}\mathbb{E}_\mathbf{S}\|\mathbf{a}\|  \\
&\leq &{L}_{j}\|\mathbf{a}_k\|+\mathbb{E}_\mathbf{S} \|r_j(\mathbf{S},0)\|\\\nonumber
&\leq &{L}_{j}\|\mathbf{a}_k\|+\beta_2
  \end{eqnarray}
   Where $\beta_2\triangleq\mathbb{E}_\mathbf{S} \|r_j(\mathbf{S},0),$  ${L}_{j}\triangleq\mathbb{E}_\mathbf{S}L_{j,\mathbf{S}}.$ The above set of inequalities  \ref{Q2} show that $\mathbb{E}_\mathbf{S} \|r_j(\mathbf{S},\mathbf{a})\|$ is bounded.

   Combining the results of inequalities in  \ref{Q1}  \ref{Q2} we can get

    \begin{eqnarray*}
 \|M_{j,k+1}\|^2 &\leq& z^2 b^2  (\beta_{1,\mathbf{S}}+L_{j,\mathbf{S}}\|\mathbf{a}_k\|+ {L}_{j}\|\mathbf{a}_k\|+\beta_2)^2\\
    &\leq& 2 z^2 b^2((\beta_{1,\mathbf{S}}+\beta_2)^2+(L_{j,\mathbf{S}}+ {L}_{j})^2\|\mathbf{a}_k\|^2)\\
    &\leq& 4 z^2 b^2(\beta_{1,\mathbf{S}}^2+\beta_2^2+(L_{j,\mathbf{S}}^2+ {L}_{j}^2)\|\mathbf{a}_k\|^2)
%    &\leq& 4 z^2 b^2(\beta_{1,\mathbf{S}}^2+\beta_2^2+(L_{j,\mathbf{S}}^2+ {L}_{j}^2)\|\mathbf{a}_k\|^2)
   \end{eqnarray*}

Taking $\mathbb{E}_\mathbf{S}$ over the above inequalities we get:
    \begin{eqnarray*}
\mathbb{E}_\mathbf{S} \|M_{j,k+1}\|^2
&\leq& 4 z^2 b^2(\mathbb{E}_\mathbf{S}\beta_{1,\mathbf{S}}^2+\beta_2^2+(\mathbb{E}_\mathbf{S}L_{j,\mathbf{S}}^2+ {L}_{j}^2)\|\mathbf{a}_k\|^2)\\
%&\leq& 4 z^2 b^2(\mathbb{E}_\mathbf{S}\beta_{1,\mathbf{S}}^2+\beta_2^2+(\mathbb{E}_\mathbf{S}L_{j,\mathbf{S}}^2+ {L}_{j}^2)\|\mathbf{a}_k\|^2)\\
%&\leq& 4 z^2 b^2 \mathbb{E}_\mathbf{S}  ((\beta_{1,\mathbf{S}}+\beta_2)^2+(L_{j,\mathbf{S}}+ {L}_{j})^2\|\mathbf{a}_k\|^2)\\
%&\leq& z^2 b^2  (\beta_{1,\mathbf{S}}+\beta_2+(L_{j,\mathbf{S}}+ {L}_{j})\|\mathbf{a}_k\|)^2\\
&\leq&4 z^2 b^2  (\beta+\grave{L}_j\|\mathbf{a}_k\|^2)\\
&\leq& \acute{c} (1+\|\mathbf{a}_k\|^2)\\
%&\leq& z^2 b^2 \mathbb{E}_\mathbf{S} [ (L_{j,\mathbf{S}}\|\mathbf{a}\|^2] +{L}_{j}\mathbb{E}_\mathbf{S}\|\mathbf{a}\|^2\\
   \end{eqnarray*}
Where $\grave{L}_j\triangleq \mathbb{E}_\mathbf{S}L_{j,\mathbf{S}}^2+ {L}_{j}^2$ , $\beta \triangleq \mathbb{E}_\mathbf{S}\beta_{1,\mathbf{S}}^2+\beta_2^2$ and $\acute{c}\geq 4 z^2 b^2 (\beta+ \grave{L}_j) $

This completes the proof.
\end{proof}

We now combine the above three steps to derive almost sure convergence to an ODE.
To do so, we interpolate the stochastic process $\mathbf{a}_k$ (an affine interpolation) in order to get a continuous time process following the lines of Borkar \cite{borkar} Chapter 2 Lemma 1. The gap between the solution of the non-autonomous differential equation given by
$$
\frac{d}{dt}\mathbf{a}_t=f(t,\mathbf{a}_t)
$$
and the interpolated process vanishes almost surely for asymptotic interval of length $T>0.$
\begin{eqnarray*}
    \lim_{t\longrightarrow\infty}\sup_{q\in[t,t+T]}\|\tilde{\mathbf{a}}_q-\mathbf{a}_q^*\|&=&0 \; a.s.
\end{eqnarray*}
In order to calculate the bound we need to define a few terms which are helpful in obtaining a compact form of the bound.
 \begin{eqnarray*}
\sup_{t\in[t_k,t_k+T]}
    \|\tilde{\mathbf{a}}(t)-\mathbf{a}^{t_k}(t)\|&\leq& K_{T,t}e^{LT}+C_T\lambda_{t+\acute{k}}\\
                            &=& C_T (\lambda_{t+\acute{k}}+L\sum_{\acute{k}\geq0}\lambda_{t+\acute{k}}^2 )\\
                            &&+\sup_{\acute{k}\geq 0}\|\delta_{t,t+\acute{k}}\|
  \end{eqnarray*}
where
  \begin{eqnarray*}
  K_{T,t} & \triangleq & C_T L \sum_{\acute{k}\geq0}\lambda_{t+\acute{k}}^2+\sup_{\acute{k}\geq0}\|\delta_{t,t+\acute{k}}\|\\
  \delta_{t,t+\acute{k}}&\triangleq&\xi_{t+\acute{k}}-\xi_{t}\\
  \xi_{t}&\triangleq&\sum_{m=0}^{t-1}\lambda_m M_{m+1}\\
  C_T&\triangleq&\|r(0)\|+L(C_0+\|r(0)\|T)e^{LT}<\infty\\
    \mathbb{P}\left(\sup_{\acute{k}}  \|\mathbf{a}_{\acute{k}}\| <C_0\right)&=&1 \\
  \end{eqnarray*}

then we conclude  by discrete adaptation of Lemma 1 in Borkar \cite{borkar}.

\subsection{Fixed Step Size: Proof of Theorem \ref{Theorem2}}

Theorem  \ref{Theorem2} states that
    Under Assumption A2, the learning algorithm converges in distribution to the trajectory of a non-autonomous system given by

    \begin{eqnarray*}
           \frac{d}{dt}\hat{{a}}_{j,t}&=& z_j b_j \sin(\Omega_j t+\phi_j)\mathbb{E}_{{\mathbf{S}}}\left(r_j( {\mathbf{S}}, {a}_t)\right)\\
           {a}_{j,t}&=&\hat{{a}}_{j,t}+b_j \sin(\Omega_j t+\phi_j)
  \end{eqnarray*}

\begin{prop} \label{labelProposition1} Let $\tilde{\hat{\mathbf{a}}}_t$ be the interpolated version the trajectory of our algorithm at time $t$ ${\hat{\mathbf{a}}}_t$ is the trajectory of the the ODE at time $t$. Under assumption A2 $\tilde{\hat{\mathbf{a}}}_t$ converges to $\hat{\mathbf{a}}_t$ as step size vanishes.

    $$
     \mathbb{E}\sup_{t\in[0,T]}[\|\tilde{\hat{\mathbf{a}}}_t-\hat{\mathbf{a}}_t\|^2]^{\frac{1}{2}}=\tilde{C}_T\sqrt{\lambda}
    $$
    \end{prop}
    Proposition  \ref{labelProposition1} implies theorem  \ref{Theorem2}.

%Proof:
\begin{proof}[Proof of Proposition  \ref{labelProposition1}]
To prove the above proposition we start with a fixed step size $\lambda>0.$

\begin{itemize}
  \item Time Scale. $T_t=\sum_{k=1}^t\lambda=t\lambda$, for $t\geq0$
  \item The cumulative noise at iteration $t$ $\xi_t=\sum_{k=1}^{t-1}\lambda M_{k+1}=\lambda\sum_{k=1}^{t-1} M_{k+1}$
  \item Define the (affine) interpolated process from $\{\hat{\mathbf{a}}\}_{k\geq0}$ rewritten as
  $$\hat{{a}}_{j,k+1}=\hat{{a}}_{j,k}+\lambda(f_j(k,\hat{{a}}_{k})+M_{k+1}).$$ The advantage of the interpolated process is that it is defined for any continuous time by concatenation.
  The affine interpolation writes $\tilde{\hat{{a}}}_{j,t}=\hat{{a}}_{j,k}+(\frac{t-T_k}{\lambda})(\hat{{a}}_{j,k+1}-\hat{{a}}_{j,k})$ if $t\in[k\lambda,(k+1)\lambda[$ which is now in continuous time.
\end{itemize}

Note that constant learning rate or constant step size $\lambda_t=\lambda$ is suitable for many practical scenarios. It is used for example in numerical analysis:
    Euler-s Scheme (1st Order), Runge Kutta's scheme (4th Order), etc. Our algorithm writes

$$(**)\left\{
\begin{array}{rcl}
    \hat{{a}}_{j,k+1} &=& \hat{{a}}_{j,k}+\lambda(b_jz_j\sin(\hat{k}_t\Omega_j+\phi_j))\tilde{r}_{j,t}\\
    {a}_{j,k+1} &=& \hat{{a}}_{j,k}+(b_j\sin(\hat{k}_t\Omega_j+\phi_j))
\end{array}\right.$$
where $\lambda$ is a constant learning rate, our aim is to analyze $(**)$ asymptotically when $\lambda$ is very small.
In order to prove an asymptotic pseudo-trajectory result for constant learning rate, we need additional assumptions of the sequence generated by the powers. The key additional assumption is the uniform integrability of that process.
We need  the conditions $C_1$ $C_2$, which translate into
    \begin{description}
        \item[-] From Remark \ref{remark5}: gradient of the expectation of payoff is bounded %$A1(\mathcal{S})$
        \item[-] From Lemma \ref{lemma2}: Square of the martingale is bounded %$A3$
        \item[-] Uniform Integrability of $r_j(\mathbf{S},\mathbf{a})$
    \end{description}
    and $\hat{\mathbf{a}}_{t}$ is the solution of $\dot{\hat{\mathbf{a}}}_{t}=f(t,\hat{\mathbf{a}}_{t})$ starting from $\hat{\mathbf{a}}_{\lfloor\frac{t}{\lambda}\rfloor}$

 \begin{eqnarray*}
    \tilde{\hat{\mathbf{a}}}_{T_{t+m}}&=&\sum_{k=1}^{m}\underbrace{(\tilde{\hat{\mathbf{a}}}_{T_{t+k}}-\tilde{\hat{\mathbf{a}}}_{T_{t+k-1}})}_{\lambda(f_j(\lfloor\frac{T_{t+k-1}}{\lambda}\rfloor,\hat{\mathbf{a}}_{\lfloor\frac{T_{t+k-1}}{\lambda}\rfloor})+M_{k+1})}\\
    &&+\tilde{\hat{\mathbf{a}}}_{\lfloor\frac{T_{t+k-1}}{\lambda}\rfloor}\\
    \tilde{\hat{\mathbf{a}}}_{T_{t+m}}&=&\sum_{k=1}^{m}(T_{t+k}-T_{t+k-1}){ f_j(\lfloor\frac{T_{t+k-1}}{\lambda}\rfloor,\hat{\mathbf{a}}_{\lfloor\frac{T_{t+k-1}}{\lambda}\rfloor})}\\&&
    +\sum_{k=1}^{m}\lambda M_{t+k+1} +\tilde{\hat{\mathbf{a}}}_{T_{t}}\\
    \tilde{\hat{\mathbf{a}}}_{T_{t+m}}&=&\sum_{k=1}^{m}\int_{T_{t+k}}^{T_{t+k-1}}{ f_j(\lfloor\frac{T_{t+k-1}}{\lambda}\rfloor,\hat{\mathbf{a}}_{\lfloor\frac{T_{t+k-1}}{\lambda}\rfloor})}ds\\
    &&+\sum_{k=1}^{m}\lambda M_{t+k+1}+\tilde{\hat{\mathbf{a}}}_{T_{t}}\\
      \end{eqnarray*}

      \begin{eqnarray*}
    \tilde{\hat{\mathbf{a}}}_{T_{t+m}}&=&\sum_{k=1}^{m}\int_{T_{t+k}}^{T_{t+k-1}}{ f_j(\lfloor\frac{T_{t+k-1}}{\lambda}\rfloor,\hat{\mathbf{a}}_{\lfloor\frac{T_{t+k-1}}{\lambda}\rfloor})}ds\\
    &&+ (\xi_{t+m}- \xi_{t})+\tilde{\hat{\mathbf{a}}}_{T_{t}}\\
    \tilde{\hat{\mathbf{a}}}_{T_{t+m}}&=&\sum_{k=1}^{m}\int_{T_{t+k}}^{T_{t+k-1}}{ f_j(\lfloor\frac{s}{\lambda}\rfloor,\tilde{\hat{\mathbf{a}}}_{\lfloor\frac{s}{\lambda}\rfloor})}ds\\
    &&+(\xi_{t+m}- \xi_{t})+\tilde{\hat{\mathbf{a}}}_{T_{t}}
  \end{eqnarray*}

Now we use Burkholder's inequality which states the following: For an  $\alpha>0$  there exists two constants $c_1>0$ and  $c_2>0$  such that

\begin{eqnarray*}
c_1 \mathbb{E}[\sum_{k=1}^t\|\hat{\mathbf{a}}_{k}-\hat{\mathbf{a}}_{k-1}\|^{2}]^{\alpha/2}&\leq &\mathbb{E} [\sup_{m\geq t} \|\hat{\mathbf{a}}_{t}\|]
\\&\leq& c_2\mathbb{E}[\sum_{k=1}^t\|\hat{\mathbf{a}}_{k}-\hat{\mathbf{a}}_{k-1}\|^{2}]^{\alpha/2}
  \end{eqnarray*}

\begin{eqnarray*}
c_1 \mathbb{E}[\sum_{k=1}^t\|\hat{\eta}_{k}-\hat{\eta}_{k-1}\|^{2}]^{\alpha/2}&\leq &\mathbb{E} [\sup_{k\leq t} \|{\eta}_{k}\|]
\\&\leq& c_2\mathbb{E}[\sum_{k=1}^t\|\hat{\eta}_{k}-{\eta}_{k}\|^{2}]^{\alpha/2}
  \end{eqnarray*}
Take $\hat{\eta}_t=\lambda\sum_{k\leq m}\|M_{t+k}\|^2$ and we use discrete Gronwall inequality which states that
\begin{eqnarray*}
\epsilon_{t+1}&\leq& C+L\sum_{k=0}^t \lambda \epsilon_k\\
\epsilon_{t+1}&\leq& Ce^{L\lambda t}
  \end{eqnarray*}
  where $\epsilon_t>0 \; \forall t\geq0$

  for $$\epsilon_k=\mathbb{E}[\sup_{k'\leq k}\|\tilde{\hat{\mathbf{a}}}_{T_{t+k'}}-\hat{\mathbf{a}}_{T_{t+k'}}\|^2]^{1/2}$$
  $C=\lambda T K_1 \sqrt{1+C_0^2}+\sqrt{\lambda K_2 (1+C_0^2)},$ $L=\max_{j\in \mathcal{N}} \mathbb{E}_\mathbf{S}[L_{j,\mathbf{S}}]$ for some $K_1, K_2=c_2$

  from the above we deduce that

  $$
  \mathbb{E}[\sup_{k'\leq k}\|\tilde{\hat{\mathbf{a}}}_{T_{t+k'}}-\hat{\mathbf{a}}_{T_{t+k'}}\|^2]^{1/2}\leq\sqrt{\lambda}C_T
  $$

  $K_1=\max(c_1,c_2\sqrt{1+C_0^2})$

  This shows that $\mathbb{E}[\sup_{k'\leq k}\|\tilde{\hat{\mathbf{a}}}_{T_{t+k'}}-\hat{\mathbf{a}}_{T_{t+k'}}\|^2]^{1/2}$ is bounded and implies Proposition \ref{Theorem1}. When $\lambda\longrightarrow0$ we have a weak convergence of the interpolated process to a solution of the ODE. The error gap is $\sqrt{\lambda}C_T$ which vanishes as $\lambda\longrightarrow0$.

%    $\tilde{\hat{\mathbf{a}}}_{T_{t+m}}=\sum_{k=1}^{m}\underbrace{(\tilde{\hat{\mathbf{a}}}_{T_{t+k}}-\tilde{\hat{\mathbf{a}}}_{T_{t+k-1}})}_{\lambda(f_j(\lfloor\frac{T_{t+k-1}}{\lambda}\rfloor,\hat{\mathbf{a}}_{\lfloor\frac{T_{t+k-1}}{\lambda}\rfloor})+M_{k+1})}+\tilde{\hat{\mathbf{a}}}_{\lfloor\frac{T_{t+k-1}}{\lambda}\rfloor}$\\

%    $\tilde{\hat{\mathbf{a}}}_{T_{t+m}}=\sum_{k=1}^{m}(T_{t+k}-T_{t+k-1}){\lambda f_j(\lfloor\frac{T_{t+k-1}}{\lambda}\rfloor,\hat{\mathbf{a}}_{\lfloor\frac{T_{t+k-1}}{\lambda}\rfloor})}+\sum_{k=1}^{m}\lambda M_{k+1}+\tilde{\hat{\mathbf{a}}}_{T_{t}}$
%
% $\tilde{\hat{\mathbf{a}}}_{T_{t+m}}=\sum_{k=1}^{m}\int_{T_{t+k}}^{T_{t+k-1}}{\lambda f_j(\lfloor\frac{T_{t+k-1}}{\lambda}\rfloor,\hat{\mathbf{a}}_{\lfloor\frac{T_{t+k-1}}{\lambda}\rfloor})}ds+\sum_{k=1}^{m}\lambda M_{k+1}+\tilde{\hat{\mathbf{a}}}_{T_{t}}$
\end{proof}

\section{Conditions for our Example}

Following are some details about how to obtain $\mathbf{a}^*$ for our application.

$$
    g_{i,j}\triangleq |h_{i,j}|^2
$$

$$
\bar{g}_{i,j}\triangleq \mathbb{E}_\mathbf{g}g_{i,j}=\mathbb{E}_\mathbf{g}| h_{i,j}|^2
$$

%$  \mathbb{E}_\mathbf{S}\frac{\partial r_j(\mathbf{S},\mathbf{a}^*)}{\partial a_j} =0,$

 From remark \ref{remark5} we can write $ \mathbb{E}_\mathbf{G}\frac{\partial \gamma_j(\mathbf{G},\mathbf{a}^*)}{\partial a_j} = \frac{\partial }{\partial a_j}\mathbb{E}_\mathbf{G}\gamma_j(\mathbf{G},\mathbf{a}^*) =0.$  Solving  $N$ equations we have the following matrix form.

\[
    \mathbf{a}^*=\begin{pmatrix}
      \mathbf{a}_1^* \\
      \mathbf{a}_2^* \\
      \vdots  \\
      \mathbf{a}_N^* \\
     \end{pmatrix},
    \bar{\mathbf{G}}=
     \begin{pmatrix}
          \bar{g}_{1,1} & \bar{g}_{1,2} & \cdots & \bar{g}_{1,N} \\
          \bar{g}_{2,1} & \bar{g}_{2,2} & \cdots & \bar{g}_{2,N} \\
          \vdots  & \vdots  & \ddots & \vdots  \\
          \bar{g}_{N,1} & \bar{g}_{N,2} & \cdots & \bar{g}_{N,N}
     \end{pmatrix},
     \bar{\mathbf{a}}=
     \begin{pmatrix}
          \frac{\omega \bar{g}_{1,1}}{\lambda}-\sigma^2 \\
          \frac{\omega \bar{g}_{2,2}}{\lambda}-\sigma^2 \\
          \vdots  \\
          \frac{\omega \bar{g}_{N,N}}{\lambda}-\sigma^2 \\
     \end{pmatrix}
\]

The above equation can be written in the compact form as

\[
    \mathbf{a}^*=\bar{\mathbf{G}}^{-1}\bar{\mathbf{a}}
\]

$\bar{\mathbf{G}}$ should be invertible and  all the elements in the vector $\bar{\mathbf{a}}$ should be strictly positive as they are a linear combination of power and gains which are positive. We can also write $ \omega \bar{g}_{j,j}>\lambda\sigma^2 $.
For this example we can write

$$
    \mathbb{E}_\mathbf{G}[g_{j,j}] > \sum_{j'\neq j}\mathbb{E}_\mathbf{G}[g_{j,j'}] \; \forall j, \;j'\neq j
$$
If this condition is satisfied then $\bar{\mathbf{G}}$ is invertible.

As $G$ is a matrix of random channel gains it is almost surely invertible. To show the invertibility of this matrix we just need to show that the $det(\bar{\mathbf{G}})\neq0$

\end{document}